\documentclass[11pt]{article}

\usepackage[margin=1.2in]{geometry}

\usepackage[T1]{fontenc}
\usepackage[utf8]{inputenc}
\usepackage[USenglish, UKenglish, english]{babel}
\usepackage{amssymb}
\usepackage{amsthm}
\usepackage{amsmath}
\usepackage[mathscr]{euscript}
\usepackage{enumitem}
\usepackage{graphicx}
\usepackage{MnSymbol}
 \let\mathscr\relax
\usepackage[scr]{rsfso}
\usepackage{tikz-cd}
\usepackage{tikz}
\usetikzlibrary{matrix,arrows,decorations.pathmorphing}
\usepackage{hyperref}
\usepackage{marginnote}
\usetikzlibrary{arrows.meta}

\usepackage{cases}

\usepackage{libertine}

\theoremstyle{definition}
\newtheorem{defin}{Definition}[section]
\theoremstyle{definition}

\theoremstyle{plain}
\newtheorem{theo}[defin]{Theorem}
\theoremstyle{plain}
\newtheorem{prop}[defin]{Proposition}
\theoremstyle{plain}
\newtheorem{lem}[defin]{Lemma}
\theoremstyle{plain}
\newtheorem{cor}[defin]{Corollary}
\theoremstyle{definition}
\newtheorem{rmk}[defin]{Remark}
\theoremstyle{definition}

\theoremstyle{definition}

\theoremstyle{plain}
\newtheorem{conj}[defin]{Conjecture}
\theoremstyle{definition}

\theoremstyle{definition}

\theoremstyle{plain}
\newtheorem{dade}[defin]{Dade's Projective Conjecture}

\theoremstyle{plain}
\newtheorem{alp}[defin]{Alperin's Weight Conjecture}

\theoremstyle{plain}
\newtheorem{ctc}[defin]{Character Triple Conjecture}

\theoremstyle{definition}

\theoremstyle{plain}

\theoremstyle{definition}

\theoremstyle{definition}
\newtheorem*{defin*}{Definition}
\theoremstyle{definition}
\newtheorem*{ex*}{Example}
\theoremstyle{plain}
\newtheorem*{theo*}{Theorem}
\theoremstyle{plain}
\theoremstyle{plain}
\newtheorem*{conj*}{Conjecture}
\newtheorem*{prop*}{Proposition}
\theoremstyle{plain}
\newtheorem*{lem*}{Lemma}
\theoremstyle{plain}
\newtheorem*{cor*}{Corollary}
\theoremstyle{definition}
\newtheorem*{rmk*}{Remark}
\theoremstyle{definition}
\newtheorem*{exe*}{Exercise}

\theoremstyle{plain}
\newtheorem{theoA}{Theorem}

\theoremstyle{plain}

\theoremstyle{plain}

\theoremstyle{plain}

\theoremstyle{plain}
\newtheorem{corA}[theoA]{Corollary}

\numberwithin{equation}{section}
\usepackage{parskip}

\makeatletter
\def\thm@space@setup{%
  \thm@preskip=\parskip \thm@postskip=0pt
}
\makeatother

\setlist[enumerate]{label=(\roman*)}
\def\bl{{\rm bl}}

\def\irr{{\rm Irr}}

\def\ab{{\rm Ab}}

\def\aut{{\rm Aut}}

\def\GL{{\rm GL}}

\def\n{{\mathbf{N}}} %normalizer

\def\c{{\mathbf{C}}} %centralizer
 
\def\z{{\mathbf{Z}}} %center

\def\O{{\mathbf{O}}} %radical

\def\E{{\mathcal{E}}} %Lusztig series

\def\C{{\mathcal{C}}} %couples chain-character

\def\k{{k}}

\def\G{{\mathbf{G}}} %algebraic group

\def\H{{\mathbf{H}}} %algebraic subgroup

\def\K{{\mathbf{K}}} %algebraic subgroup

\def\L{{\mathbf{L}}} %Levi

\def\M{{\mathbf{M}}} %Levi

\def\T{{\mathbf{T}}} %Torus

\def\S{{\mathbf{S}}} %Torus

\def\CL{{\mathcal{L}}} %couples (chain of Levis)-character

\DeclareMathOperator{\CP}{\mathcal{C}\CPkern \mathcal{P}}
\newcommand{\CPkern}{%
  \mkern-1.0mu
  \mathchoice{}{}{\mkern0.2mu}{\mkern0.5mu}%
}
\makeatletter
\newcommand{\uset}[3][0ex]{%
  \mathrel{\mathop{#3}\limits_{
    \vbox to#1{\kern-7\ex@
    \hbox{$\scriptstyle#2$}\vss}}}}
\makeatother

\newcommand{\ws}[1][1.5]{
  \mathrel{\scalebox{#1}[1]{$\sim$}}
}

\newcommand{\iso}[1]{\ws_{#1}}

\usepackage{sectsty}
\allsectionsfont{\centering}

\usepackage{xparse,etoolbox}

\newcommand{\blocktheorem}[1]{%
  \csletcs{old#1}{#1}% Store \begin
  \csletcs{endold#1}{end#1}% Store \end
  \RenewDocumentEnvironment{#1}{o}
    {\par\addvspace{1.5ex}
     \noindent\begin{minipage}{\textwidth}
     \IfNoValueTF{##1}
       {\csuse{old#1}}
       {\csuse{old#1}[##1]}}
    {\csuse{endold#1}
     \end{minipage}
     \par\addvspace{1.5ex}}
}

\blocktheorem{theoA}
\blocktheorem{conjA}
\blocktheorem{paraA}
\blocktheorem{corA}

\makeatletter
\def\blfootnote{\gdef\@thefnmark{}\@footnotetext}
\makeatother

\title{
{\huge\bf The Brown complex in non-defining characteristic and applications}\\
\author{\Large Damiano Rossi}
\blfootnote{\emph{$2010$ Mathematical Subject Classification:} $20$J$05$, $20$G$40$, $55$P$91$, $20$C$20$
\\
\emph{Key words and phrases:} Brown complex, finite reductive groups, generic Sylow theory, Dade's conjecture.
\\
\date{}
I would like to thank Francesco Fumagalli for some comments on an earlier version of this paper and Gunter Malle for a thorough reading of the material presented here and for pointing out some inaccuracies. This work was initiated during a visit of the author to the University of Florence and then completed during a stay at the Mathematisches Forschungsinstitut Oberwolfach funded by a Leibniz Fellowship. I am grateful to both institutions for their hospitality and support.
}}

\begin{document}

\renewcommand{\thetheoA}{\Alph{theoA}}

\renewcommand{\thecorA}{\Alph{corA}}

\selectlanguage{english}

\maketitle

\begin{abstract}
We study the Brown complex associated to the poset of $\ell$-subgroups in the case of a finite reductive group defined over a field $\mathbb{F}_q$ of characteristic prime to $\ell$. First, under suitable hypotheses, we show that its homotopy type is determined by the generic Sylow theory developed by Brou\'e and Malle and, in particular, only depends on the multiplicative order of $q$ modulo $\ell$. This result leads to several interesting applications to generic Sylow theory, mod $\ell$ homology decompositions, and $\ell$-modular representation theory. Then, we conduct a more detailed study of the Brown complex in order to establish an explicit connection between the local-global conjectures in representation theory of finite groups and the generic Sylow theory. This is done by isolating a family of $\ell$-subgroups of finite reductive groups that corresponds bijectively to the structures controlled by the generic Sylow theory. 
\end{abstract}

\section*{Introduction}

The Brown complex $\Delta(\mathcal{S}_\ell^\star(G))$ associated to the poset of non-trivial $\ell$-subgroups of a finite group $G$, where $\ell$ is a prime dividing the order of $G$, was first introduced by K. S. Brown in his work on the Euler characteristic \cite{Bro75}. Its study led to important contributions to $\ell$-local group theory, mod $\ell$ group cohomology, and $\ell$-modular representation theory. In his seminal paper \cite{Qui78}, Quillen presented the first systematic study of the Brown complex and proved that its homotopy type coincides with that of (the simplicial complex of) closely related sets of subgroups. In particular, for the case of a finite reductive group with $\ell$ equal to the defining characteristic, he showed that the Brown complex is homotopy equivalent to the Tits building associated to the set of proper rational parabolic subgroups \cite{Tit74}. In this paper we are interested in the Brown complex of a finite reductive group in the remaining case where $\ell$ is different from the defining characteristic.

Our first theorem provides a description of the homotopy type of the Brown complex in terms of the generic Sylow theory developed by Brou\'e and Malle in \cite{Bro-Mal92}. Let $\G$ be a linear algebraic group defined over an algebraically closed field $\mathbb{F}$ of characteristic $p$ and assume that $\G$ is connected and reductive. Let $F:\G\to \G$ be a Frobenius endomorphism defining an $\mathbb{F}_q$-structure on (the variety) $\G$ for some power $q$ of $p$. The set of $\mathbb{F}_q$-rational points $\G^F$, or sometimes the pair $(\G,F)$, is called a \textit{finite reductive group}. Fix a prime $\ell$ different form $p$ and denote by $e_{\ell}(q)$ the multiplicative order of $q$ modulo $\ell$. Using the generic Sylow theory we define the simplicial complex $\Delta(\CL_{e_{\ell}(q)}^\star(\G,F))$ associated to the poset of proper $e_{\ell}(q)$-split Levi subgroups of $(\G,F)$ (see Section \ref{sec:Order polynomial} for a precise definition).

\begin{theoA}
\label{thm:Main homotopy equivalence for ell-subgroups and e-split Levi}
Let $(\G,F)$ be a finite reductive group and consider a prime number $\ell\in\pi(\G,F)$ (see Definition \ref{def:Prime condition}) not dividing the order of $\z(\G)^F$. Then there exists a $\G^F$-homotopy equivalence
\[\Delta\left(\mathcal{S}_\ell^\star\left(\G^F\right)\right)\simeq_{\G^F}\Delta\left(\CL_{e_{\ell}(q)}^\star\left(\G,F\right)\right).\]
\end{theoA}

The above theorem tells us, in particular, that the homotopy type of the Brown complex in non-defining characteristics is generic, that is, it does not depend on $\ell$ but only on the positive integer $e_{\ell}(q)$. In this way, we obtain equivalent local structures at different primes $\ell$ and $\ell'$ for a fixed finite reductive group $(\G,F)$. We mention that, in a somewhat opposite direction, work of Broto, M{\o}ller and Oliver \cite{Bro-Mol-Oli} shows that there exist different finite reductive groups $(\G,F)$ and $(\G',F')$ having (isotypically) equivalent fusion systems at a fixed prime $\ell$.

Before discussing the motivation behind Theorem \ref{thm:Main homotopy equivalence for ell-subgroups and e-split Levi}, we derive some interesting corollaries of this result. First, inspired by Brown's homological Sylow theorem \cite[Corollary 2]{Bro75}, which states that the Euler characteristic of the Brown complex is congruent to $1$ modulo the order of a Sylow $\ell$-subgroup, we obtain the following congruence for the Euler characteristic of the simplicial complex of $e_{\ell}(q)$-split Levi subgroups. This should be interpreted as a homological version of Brou\'e--Malle's generic Sylow theorem \cite[Theorem 3.4 (4)]{Bro-Mal92}.

\begin{corA}
\label{cor:Main generic homological Sylow}
Let $(\G,F)$ be a finite reductive group and suppose that the prime $\ell$ is large for $(\G,F)$ (see \cite[Definition 5.1]{Bro-Mal-Mic93}) and does not divide the order of $\z(\G)^F$. Then
\[\chi\left(\Delta\left(\CL_{e_{\ell}(q)}^\star\left(\G,F\right)\right)\right)\equiv 1 \enspace\left({\rm mod}\enspace\left(\Phi_{e_{\ell}(q)}(q)^a\right)_\ell\right)\]
where $a$ is the $\Phi_{e_{\ell}(q)}(x)$-valuation of the order polynomial $P_{(\G,F)}(x)$ (as defined in Section \ref{sec:Order polynomial}) and $n_\ell$ denotes the largest power of $\ell$ dividing an integer $n$.
\end{corA}

Next, we consider an application to mod $\ell$ homology decompositions. According to the work of Dwyer \cite{Dwy97}, there are three standard ways to reconstruct the classifying space of a finite group, up to mod $\ell$ homology, by gluing together classifying spaces of subgroups: namely the centraliser decomposition, the subgroup decomposition, and the normaliser decomposition. All three are guaranteed by a property known as \textit{ampleness}. Once the latter is established, it is often useful to ask whether any of these three decomposition is \textit{sharp}. This property yields a closed formula for the homology of the classifying space in terms of that of classifying spaces of subgroups. We refer the reader to Section \ref{sec:Homology decompositions} and the references therein. Sharpness was extensively studied by Grodal and Smith \cite{Gro-Smi06} and obtained for numerous families of subgroups. Thanks to Theorem \ref{thm:Main homotopy equivalence for ell-subgroups and e-split Levi}, we can show some of these properties for the family of $e_{\ell}(q)$-split Levi subgroups.

\begin{corA}
\label{cor:Main sharpness}
Suppose that $\ell\in\pi(\G,F)$ does not divide the order of $\z(\G)^F$. Then, the $\G^F$-simplicial complex $\Delta(\CL_{e_{\ell}(q)}^\star(\G,F))$ is ample and normaliser sharp (with respect to the prime $\ell$).
\end{corA}

The above corollary shows that the mod $\ell$ homology of the classifying space of a finite reductive group is determined $e_{\ell}(q)$-locally, that is, by the normalisers of (proper) $e_{\ell}(q)$-split Levi subgroups. More generally, the homotopy equivalence of Theorem \ref{thm:Main homotopy equivalence for ell-subgroups and e-split Levi} allows us to transfer questions involving the $\ell$-structure of the group $\G^F$ to the more convenient language of generic Sylow theory and hence to apply powerful techniques from algebraic geometry. This provides a more conceptual explanation of a connection established in previous works on $\ell$-modular representation theory (see, for instance, \cite{Bro-Mal-Mic93}, \cite{Cab-Eng94}, \cite{Cab-Eng99}, and \cite{Kes-Mal15}). Our main motivation is to reformulate some important open conjectures in representation theory of finite groups in terms of $e_{\ell}(q)$-local structures. This is particularly explicative if we consider Alperin's Weight Conjecture in the form introduced by Kn\"orr and Robinson \cite{Kno-Rob89}. According to work of Webb \cite{Web87} and Thev\'enaz \cite{The93} the latter can be stated for a given finite group $G$ using the \textit{reduced Lefschetz invariant} $\tilde{\Lambda}_G$ and the \textit{equivariant Euler characteristic} $\chi_G$ respectively (see Section \ref{sec:AWC} for more details). Since these are invariant under $G$-homotopy equivalences we can use Theorem \ref{thm:Main homotopy equivalence for ell-subgroups and e-split Levi} to deduce the following corollary.

\begin{corA}
\label{cor:Main AWC}
Let $(\G,F)$ be a finite reductive group and consider a prime number $\ell\in\pi(\G,F)$ (see Definition \ref{def:Prime condition}) not dividing the order of $\z(\G)^F$. Then, the block-free version of the Kn\"orr--Robinson reformulation of Alperin's Weight Conjecture holds for $\G^F$ if and only if
\[k^0\left(\G^F\right)=-l_\ell\left(\tilde{\Lambda}_{\G^F}\left(\Delta\left(\CL_{e_{\ell}(q)}^\star(\G,F)\right)\right)\right)\]
if and only if
\[k\left(\G^F\right)-k^0\left(\G^F\right)=\chi_{\G^F}\left(\Delta\left(\CL_{e_{\ell}(q)}^\star(\G,F)\right)\right)\]
where the functions $k$, $k^0$ and $l_\ell$ count the number of irreducible ordinary characters, $\ell$-defect zero characters, and $\ell$-Brauer characters as explained in Section \ref{sec:AWC}.
\end{corA}

We now come to Dade's Conjecture \cite[Conjecture 6.3]{Dad92}. A long term plan for the solution of this conjecture was initiated by the author in \cite{Ros-Generalized_HC_theory_for_Dade} relying on the reduction theorem of Sp\"ath \cite{Spa17} and inspired by old ideas of Brou\'e, Fong and Srinivasan \cite{Bro-Fon-Sri1}. In particular, in \cite[Conjecture C]{Ros-Generalized_HC_theory_for_Dade} the author introduced a version of Dade's Conjecture for finite reductive groups adapted to generalised Harish-Chandra theory by using the $e_{\ell}(q)$-local structures arising from the simplicial complex of $e_{\ell}(q)$-split Levi subgroups. Now, in order to prove Dade's Conjecture we need to show that:
\begin{itemize}
\item \cite[Conjecture C]{Ros-Generalized_HC_theory_for_Dade} holds for finite reductive groups; and
\item \cite[Conjecture C]{Ros-Generalized_HC_theory_for_Dade} implies Dade's Conjecture.
\end{itemize}
The first question was reduced to certain extendibility conditions in \cite{Ros-Clifford_automorphisms_HC} and proved for unipotent characters and groups of types ${\bf A}$, ${\bf B}$, and ${\bf C}$ in \cite{Ros-Unip}. The second question, was answered affirmatively in \cite[Proposition 7.10]{Ros-Generalized_HC_theory_for_Dade} under the assumption that $\ell$ is large. Our next theorem extends this result to primes $\ell$ satisfying the hypotheses of Theorem \ref{thm:Main homotopy equivalence for ell-subgroups and e-split Levi}. Unfortunately, in this case the homotopy equivalence constructed above is not enough to obtain the desired result. The proof relies instead on the construction of a series of \textit{alternations} (see Definition \ref{def:Alternations}) used to prove a cancellation theorem (see Theorem \ref{thm:Cancellation theorem for Dade}) that reduces the study of Dade's Conjecture from the Brown complex to the simplicial complex of \textit{$e_{\ell}(q)$-closed abelian $\ell$-subgroups} (see Definition \ref{def:e-closure}). The latter is then identified with the simplicial complex of $e_{\ell}(q)$-split Levi subgroups (see Proposition \ref{prop:Replacement theorem for Dade}) and this ultimately leads to the following theorem.

\begin{theoA}
\label{thm:Main replacement theorem for Dade}
Let $(\G,F)$ be a finite reductive group with $\G$ simple of simply connected type. Consider an odd prime $\ell$ good for $\G$ and not dividing the order of $\z(\G)^F$, with $\ell\neq 3$ if $(\G,F)$ has rational type ${^3{\bf D}_4}$. Then \cite[Conjecture C]{Ros-Generalized_HC_theory_for_Dade} and Dade's Conjecture (see \cite[Conjecture 6.3]{Dad92}) are equivalent for every Brauer $\ell$-block $B$ with non trivial defect and every non-negative integer $d$.
\end{theoA}

A similar argument can be used to prove a cancellation theorem (see Theorem \ref{thm:Cancellation theorem for CTC}) for the Character Triple Conjecture \cite[Definition 6.3]{Spa17}. This is then used to show that the inductive condition for Dade's Conjecture \cite[Definition 6.7]{Spa17} follows from \cite[Conjecture D]{Ros-Generalized_HC_theory_for_Dade} and leads to our final result.

\begin{theoA}
\label{thm:Main replacement theorem for inductive Dade's condition}
Let $(\G,F)$ be a finite reductive group with $\G$ simple of simply connected type and such that $\G^F$ is the universal covering group of $\G^F/\z(\G)^F$. Consider a prime number $\ell\geq 5$ good for $\G$ and not dividing the order of $\z(\G)^F$. If \cite[Conjecture D]{Ros-Generalized_HC_theory_for_Dade} holds for a Brauer $\ell$-block $B$ with non-trivial defect and a non-negative integer $d$, then the inductive condition for Dade's Conjecture (see \cite[Definition 6.7]{Spa17}) holds for $B$ and $d$. 
\end{theoA}

The paper is organised as follows. In Section \ref{sec:Background} we introduce the main definitions and collect various results on simplicial complexes, equivariant homotopy equivalences, generic Sylow theory, and centralisers of abelian $\ell$-subgroups in finite reductive groups. Here, we also introduce the notion of \textit{alternation} (see Definition \ref{def:Alternations}) inspired by the work of Kn\"orr and Robinson \cite{Kno-Rob89}. Section \ref{sec:Homotopy equivalence} is devoted to the proof of Theorem \ref{thm:Main homotopy equivalence for ell-subgroups and e-split Levi} and its applications: Corollary \ref{cor:Main generic homological Sylow}, Corollary \ref{cor:Main sharpness}, and Corollary \ref{cor:Main AWC}. We also show that all these results can be stated by replacing the simplicial complex of $e_{\ell}(q)$-split Levi subgroups with that of $\Phi_{e_{\ell}(q)}$-tori (see Proposition \ref{prop:Homotopy equivalence for e-tori}). In Section \ref{sec:Adapting ell-groups to e-split Levi} we introduce and study the notion of \textit{$e$-closure} (see Definition \ref{def:e-closure}) and of \textit{weak $e$-closure} (see Definition \ref{def:Weak e-closure}). This is then used to construct alternations (see Proposition \ref{prop:weakly e-closed and abelian subgroups} and Proposition \ref{prop:e-closed and weakly e-closed subgroups}) inside the Brown complex of a finite reductive group and reduce the study to the $e_{\ell}(q)$-closed abelian $\ell$-subgroups. This leads us, in Section \ref{sec:Dade and CTC}, to cancellation theorems for Dade's Conjecture (see Theorem \ref{thm:Cancellation theorem for Dade}) and the Character Triple Conjecture (see Theorem \ref{thm:Cancellation theorem for CTC}) thanks to which we can state the conjectures by only using chains of $e_{\ell}(q)$-closed abelian $\ell$-subgroups. Finally, after showing that the latter correspond to chains of $e_{\ell}(q)$-split Levi subgroups (see Lemma \ref{lem:e-split and e-closed subgroups}), we are able to recover Dade's Conjecture and its inductive condition, or better the Character Triple Conjecture, from \cite[Conjecture C and Conjecture D]{Ros-Generalized_HC_theory_for_Dade}. This finally yields Theorem \ref{thm:Main replacement theorem for Dade} and Theorem \ref{thm:Main replacement theorem for inductive Dade's condition}.

\section{Background material}
\label{sec:Background}

\subsection{Subgroup complexes}
\label{sec:Notations chains}

For every poset $\mathcal{X}$ we form a simplicial complex $\Delta(\mathcal{X})$ whose $n$-simplices are the totally ordered chains $\sigma$ of cardinality $n+1$ consisting of elements of $\mathcal{X}$, that is, the chains of the form $\sigma=\{x_0<x_1<\dots<x_n\}$ with $x_i\in\mathcal{X}$. The dimension of the simplex $\sigma$ is defined as $|\sigma|:=n$ while the dimension of the simplicial complex $\Delta(X)$ is the largest dimension of one of its simplices. When considering a simplicial complex of the form $\Delta(\mathcal{X})$ we will use interchangeably the terms \textit{chain} and simplex as well as the terms \textit{length} and dimension. Now, if $\phi:\mathcal{X}\to\mathcal{Y}$ is a map of posets (i.e. order preserving), then we get an induced map of simplicial complexes $\Delta(\phi):\Delta(\mathcal{X})\to\Delta(\mathcal{Y})$. In particular, if $G$ is a finite group acting on $\mathcal{X}$ via poset automorphisms, then $\Delta(\mathcal{X})$ is a $G$-simplicial complex as defined in \cite[Definition 6.1.1]{Ben98}. In this situation we say that $\mathcal{X}$ is a $G$-poset. Furthermore a map of $G$-posets $\phi:\mathcal{X}\to \mathcal{Y}$ is a map of posets that is additionally $G$-equivariant. Observe that for every $x\in\mathcal{X}$ the subposet $\mathcal{X}_{\leq x}:=\{x'\in\mathcal{X}\mid x'\leq x\}$ induces a subcomplex $\Delta(\mathcal{X}_{\leq x})$. Furthermore, if $G$ acts on $\mathcal{X}$, then $\Delta(\mathcal{X}_{\leq x})$ is a $G_x$-simplicial complex where $G_x$ denotes the stabiliser of $x$ in $G$. This remark applies to similarly defined subposets $\mathcal{X}_{<x}$, $\mathcal{X}_{\geq x}$ etc.

In this paper we are mostly interested in the case where $\mathcal{X}$ is a poset of subgroups ordered by inclusion. Fix a prime $\ell$ and let $Z$ be a central $\ell$-subgroup of a finite group $G$. We denote by $\mathcal{S}_\ell^\star(G,Z)$ the set of all those $\ell$-subgroups $P$ of $G$ that satisfy $Z<P$ and by $\mathcal{S}_\ell(G,Z)$ the set of those $P$ that satisfy $Z\leq P$. With this notation, we have $\mathcal{S}_\ell(G,Z)=\mathcal{S}_\ell^\star(G,Z)\cup \{Z\}$. These sets of $\ell$-subgroups have a natural structure of $G$-posets under the conjugacy action of $G$ and the usual subgroup inclusion. As explained above, we obtain a $G$-simplicial complex $\Delta(\mathcal{S}_\ell^\star(G,Z))$ consisting of the chains of $\ell$-subgroups $\sigma=\{P_0<P_1<\dots<P_n\}$ with $Z<P_i$. Next, observe that $Z$ is the minimum in the poset $\mathcal{S}_\ell(G,Z)$. In this case, we introduce the convention that every chain $\sigma$ belonging to the simplicial complex $\Delta(\mathcal{S}_\ell(G,Z))$ has $Z$ as its starting term, that is, $\sigma=\{Z=P_0<P_1<\dots<P_n\}$. The complex of $\ell$-chains $\Delta(\mathcal{S}_\ell(G,Z))$ is of fundamental importance in the study of the local-global counting conjectures and has been studied in \cite{Kno-Rob89}, \cite{Dad92} and \cite{Dad94}. The reason why we consider chains starting with a fixed $\ell$-subgroup is described in the discussion given in \cite[p.193]{Dad92}. When $Z=1$, we denote $\mathcal{S}_\ell(G,1)$ and $\mathcal{S}_\ell^\star(G,1)$ simply by $\mathcal{S}_\ell(G)$ and $\mathcal{S}_\ell^\star(G)$ respectively. The simplicial complex $\Delta(\mathcal{S}_\ell^\star(G))$ was first introduced in \cite{Bro75} and \cite{Bro76} and is known as the \textit{Brown complex}. It is important to observe that while the Brown complex is usually denoted by $\mathcal{S}_\ell(G)$, in this paper we use (for technical resons) the notation introduced above and according to which $\mathcal{S}_\ell(G)$ denotes the set of all $\ell$-subgroups of $G$, including the trivial subgroup.

As it is well-known, in many situations one can replace the set of $\ell$-subgroups of $G$ with certain subsets of $\ell$-subgroups such as the set of elementary abelian $\ell$-subgroups or the set of $\ell$-radical $\ell$-subgroups (see, for instance, \cite[Section 2]{Qui78}, \cite{Bou84} and \cite[Proposition 3.3]{Kno-Rob89}). For the purpose of this paper, we will only consider the subset of abelian $\ell$-subgroups. More precisely, we denote by $\ab_{\ell,Z}(G,Z)$ the subset of $\mathcal{S}_\ell(G,Z)$ consisting of those $\ell$-subgroups $P$ of $G$ containing $Z$ and such that $P/Z$ is abelian. We also denote by $\ab_\ell(G,Z)$ the subset of $\mathcal{S}_\ell(G,Z)$ consisting of those $\ell$-subgroups $P$ containing $Z$ and such that $P$ is abelian. Clearly, $\ab_\ell(G,Z)$ is contained in $\ab_{\ell,Z}(G,Z)$ and $Z$ belongs to $\ab_\ell(G,Z)$ since it is abelian. If we consider the same definitions but only considering $\ell$-subgroups $P$ strictly containing $Z$, then we obtain $\ab_{\ell,Z}^\star(G,Z):=\ab_{\ell,Z}(G,Z)\setminus\{Z\}$ and $\ab_\ell^\star(G,Z):=\ab_\ell(G,Z)\setminus\{Z\}$. As before we form the $G$-simplicial complexes $\Delta(\ab_\ell^\star(G,Z))$ and $\Delta(\ab_{\ell,Z}^\star(G,Z))$ corresponding to the $G$-posets $\ab_\ell^\star(G,Z)$ and $\ab_{\ell,Z}^\star(G,Z)$ respectively. Notice that also in this case the $\ell$-subgroup $Z$ is the minimum of the posets $\ab_\ell(G,Z)$ and $\ab_{\ell,Z}(G,Z)$, and therefore we conventionally assume the $\ell$-chains of $\Delta(\ab_\ell(G,Z))$ and $\Delta(\ab_{\ell,Z}(G,Z))$ to start with the term $Z$. When $Z=1$, we simplify the above notation and just write $\ab_\ell(G,1)=\ab_\ell(G)$ and similarly for the other posets. We mention that, although Quillen originally considered the set of elementary abelian $\ell$-subgroups, for the purpose of this paper it is necessary to consider the larger set of all abelian $\ell$-subgroups. The reason for this choice will become apparent in Section \ref{sec:Alternations}.

\subsection{Equivariant homotopy equivalences}
\label{sec:Homotopy}

We refer the reader to \cite{The-Web91} and \cite[Section 6.4]{Ben98} for the main definitions and results on the equivariant homotopy theory of posets. In particular, see \cite[Definition 6.4.1]{Ben98} for the definition of \textit{$G$-homotopy equivalence} that we denote here by $\simeq_G$. In what follows, we say that a $G$-poset $\mathcal{X}$ is \textit{$G$-contractible} if the simplicial complex $\Delta(\mathcal{X})$ is $G$-homotopy equivalent to the trivial $G$-simplicial complex (with trivial $G$-action). We start with the following result due to Quillen and extended by Th\'evenaz and Webb.

\begin{lem}[{{\cite[Proposition 1.6]{Qui78}, \cite[Theorem 1]{The-Web91}}}]
\label{lem:Criterion for G-homotopy}
Let $G$ be a finite group and consider two $G$-posets $\mathcal{X}$ and $\mathcal{Y}$. If $\phi:\mathcal{X}\to\mathcal{Y}$ is a map of $G$-posets such that either
\begin{enumerate}
\item $\phi^{-1}(\mathcal{Y}_{\leq y})$ is $G_y$-contractible for every $y\in\mathcal{Y}$, or
\item $\phi^{-1}(\mathcal{Y}_{\geq y})$ is $G_y$-contractible for every $y\in\mathcal{Y}$,
\end{enumerate}
then the induced map of $G$-simplicial complexes
\[\Delta(\phi):\Delta\left(\mathcal{X}\right)\to\Delta\left(\mathcal{Y}\right)\]
is a $G$-homotopy equivalence.
\end{lem}

The above lemma will be our main tool to construct $G$-homotopy equivalences of simplicial complexes. This will be used together with a standard result that ensures contractibility. A $G$-poset $\mathcal{X}$ is called \textit{canonically $G$-contractible} if there exists an element $x_0\in\mathcal{X}^G$ and a map of $G$-posets $\phi:\mathcal{X}\to\mathcal{X}$ such that $x\leq\phi(x)\geq x_0$ for every $x\in \mathcal{X}$. Here $\mathcal{X}^G$ denotes the subposet of $G$-fixed elements.

\begin{lem}
\label{lem:Criterion for G-contractibility}
Let $\mathcal{X}$ be a $G$-poset. If $\mathcal{X}$ is canonically $G$-contractible, then it is $G$-contractible.
\end{lem}

\begin{proof}
According to \cite[Theorem 6.4.2]{Ben98} it suffices to show that the simplicial complex $\Delta(\mathcal{X}^H)$ is contractible for every $H\leq G$. By assumption there exists an element $x_0\in\mathcal{X}^G$ and a map of $G$-posets $\phi:\mathcal{X}\to\mathcal{X}$ satisfying $x\leq \phi(x)\geq x_0$ for every $x\in\mathcal{X}$. In particular, we obtain a map $\phi^H:\mathcal{X}^H\to\mathcal{X}^H$ and an element $x_0\in\mathcal{X}^H$ such that $x\leq \phi^H(x)\geq x_0$ for every $x\in \mathcal{X}^H$. Then, $\Delta(\mathcal{X}^H)$ is canonically contractible thanks to \cite[1.5]{Qui78} and the result follows.
\end{proof}

There is a particularly simple situation in which the condition considered above is met. We say that a $G$-poset $\mathcal{X}$ is \textit{$G$-join-contractible (via $x_0$)} if there exists an element $x_0\in\mathcal{X}^G$ such that the join $x\vee x_0$ exists in $\mathcal{X}$ for every $x\in\mathcal{X}$.

\begin{cor}
\label{cor:Criterion for G-join-contractibility}
Let $\mathcal{X}$ be a $G$-poset. If $\mathcal{X}$ is $G$-join-contractible, then it is (canonically) $G$-contractible.
\end{cor}

\begin{proof}
Consider the map $\phi:\mathcal{X}\to\mathcal{X}$ given by $\phi(x):=x\vee x_0$. Since $x_0\in\mathcal{X}^G$, we deduce that $\phi$ is $G$-equivariant and the result follows immediately from Lemma \ref{lem:Criterion for G-contractibility}.
\end{proof}

\subsection{Alternations for simplicial complexes}

In this section, we introduce the notion of \textit{alternation} for simplicial complexes. This is inspired by the work of Kn\"orr and Robinson and, in particular, by the proof of \cite[Proposition 3.3]{Kno-Rob89} (see also the proof of \cite[Theorem 9.16]{Nav18}). As we will see below, the idea is to provide a way to cancel out the contribution of a set of simplices to certain alternating sums. This often reduces questions to the study of smaller and better behaved simplicial complexes.

\begin{defin}[Alternation]
\label{def:Alternations}
Let $\Delta$ be a simplicial complex and consider a subcomplex $\Delta'\subseteq \Delta$. We say that a map
\[\phi:\Delta\setminus \Delta'\to\Delta\setminus \Delta'\]
is an \textit{alternation} of $\Delta'$ in $\Delta$ if it satisfies $\phi^2(\sigma)=\sigma$ and $|\phi(\sigma)|=|\sigma|\pm 1$ for every $\sigma\in\Delta\setminus \Delta'$. Moreover, if $\Delta$ is a $G$-simplicial complex and $\Delta'$ is a $G$-stable subcomplex of $\Delta$, then we say that $\phi$ is a \textit{$G$-alternation} of $\Delta'$ in $\Delta$ if it is additionally $G$-equivariant.
\end{defin}

In the next lemma we illustrate how the above construction can be used to reduce problems concerning a simplicial complex $\Delta$ to a subcomplex $\Delta'$. Suppose that $\Delta$ is finite. For every non-negative integer $r$ let $\Delta_r$ be the set of simplices of $\Delta$ of dimension $r$. Then, the \textit{Euler characteristic} of $\Delta$ is given by
\[\chi(\Delta):=\sum\limits_{r\geq 0}(-1)^r\left|\Delta_r\right|\]
which is a well-defined integer since $\Delta_r=\emptyset$ for every $r$ large than the (finite) dimension of $\Delta$. We also allow the Euler characteristic to be defined on an empty simplicial complex by setting $\chi(\emptyset):=0$ (this will be relevant in the proof of Corollary \ref{cor:Homological Sylow and e-split Levi}).

\begin{lem}
\label{lem:Alternations and euler characteristics}
Let $\Delta'\subseteq \Delta$ be two finite simplicial complexes. If there exists an alternation of $\Delta'$ in $\Delta$, then $\chi(\Delta)=\chi(\Delta')$.
\end{lem}

\begin{proof}
As above let $\Delta_r$ and $\Delta_r'$ be the sets of simplices of dimension $r$ in $\Delta$ and $\Delta'$ respectively. Since the dimension $d'$ of $\Delta'$ is less than or equal to the dimension $d$ of $\Delta$, we get
\begin{align}
\label{eq:Alternations and euler characteristics, 1}
\chi\left(\Delta\right)&=\sum_{r=0}^d(-1)^r\left|\Delta_r\right|
\\
&=\sum_{r=0}^{d'}(-1)^r\left|\Delta_r'\right|+\sum_{r=0}^{d}(-1)^r\left(\left|\Delta_r\right|-\left|\Delta_r'\right|\right)\nonumber
\\
&=\chi\left(\Delta'\right)+\sum_{r=0}^{d}(-1)^r\left|\Delta_r\setminus \Delta_r'\right|.\nonumber
\end{align}
Next, let $\Delta_+$ and $\Delta_-$ be the set of simplices of even and odd dimension respectively and define $\Delta_+'$ and $\Delta_-'$ analogously. If $\phi$ is an alternation of $\Delta'$ in $\Delta$, then
\[\left|\Delta_+\setminus\Delta'_+\right|=\left|\phi\left(\Delta_+\setminus\Delta'_+\right)\right|=\left|\Delta_-\setminus\Delta'_-\right|\]
and therefore
\begin{equation}
\label{eq:Alternations and euler characteristics, 2}
\sum_{r=0}^{d}(-1)^r\left|\Delta_r\setminus \Delta'_r\right|=\left|\Delta_+\setminus\Delta'_+\right|-\left|\Delta_-\setminus\Delta'_-\right|=0.
\end{equation}
Now, combining \eqref{eq:Alternations and euler characteristics, 1} and \eqref{eq:Alternations and euler characteristics, 2} we conclude that $\chi(\Delta')=\chi(\Delta)$.
\end{proof}

Next, following \cite{Kno-Rob89} we define a \textit{$G$-stable function} to be a map defined on the set of subgroups of $G$, with integer values, and that is constant on $G$-conjugacy classes of subgroups. For instance, the maps given by $f_1(H):=\k(H)$ (the number of conjugacy classes of $H$) and $f_2(H):=l_\ell(H)$ (the number of $\ell$-regular conjugacy classes of $H$) for every subgroup $H$ of $G$ are $G$-stable functions.

\begin{lem}
\label{lem:Alternation and G-stable functions}
Let $f$ be a $G$-stable function for some finite group $G$ and consider a finite dimensional $G$-simplicial complex $\Delta$ and a $G$-stable subcomplex $\Delta'\subseteq \Delta$. If there is a $G$-alternation of $\Delta'$ in $\Delta$, then
\[\sum\limits_{\sigma}(-1)^{|\sigma|}f(G_\sigma)=\sum\limits_{\sigma'}(-1)^{|\sigma'|}f(G_{\sigma'})\]
where $\sigma$ and $\sigma'$ run over representative sets for the action of $G$ on $\Delta$ and $\Delta'$ respectively. 
\end{lem}

\begin{proof}
Let us fix representative sets $\mathcal{S}$ and $\mathcal{S}'$ for the action of $G$ on $\Delta$ and $\Delta'$ respectively. Without loss of generality we may assume that $\mathcal{S}'$ is contained in $\mathcal{S}$ so that $\mathcal{T}:=\mathcal{S}\setminus \mathcal{S}'$ is a representative set for the action of $G$ on $\Delta\setminus \Delta'$. Now, if $\phi$ is a $G$-alternation of $\Delta'$ in $\Delta$, it follows that also $\phi(\mathcal{T})$ is a representative set for the action of $G$ on $\Delta\setminus\Delta'$. Moreover, observe that $|\phi(\rho)|=|\rho|\pm 1$ and $G_{\rho}=G_{\phi(\rho)}$ for every $\rho\in\mathcal{T}$. Then, since $f$ is constant on $G$-conjugacy classes of subgroups we deduce that
\[\sum\limits_{\rho\in\mathcal{T}}(-1)^{|\rho|}f\left(G_\rho\right)=\sum\limits_{\phi(\rho)\in\phi(\mathcal{T})}(-1)^{|\phi(\rho)|}f\left(G_{\phi(\rho)}\right)=-\sum\limits_{\rho\in\mathcal{T}}(-1)^{|\rho|}f\left(G_\rho\right)\]
and therefore the alternating sum on the left-hand side is zero. We conclude that
\[\sum\limits_{\sigma\in\mathcal{S}}(-1)^{|\sigma|}f(G_\sigma)=\sum\limits_{\sigma'\in\mathcal{S}'}(-1)^{|\sigma'|}f(G_{\sigma'})+\sum\limits_{\rho\in\mathcal{T}}(-1)^{|\rho|}f(G_\rho)=\sum\limits_{\sigma'\in\mathcal{S}'}(-1)^{|\sigma'|}f(G_{\sigma'})\]
as claimed in the statement.
\end{proof}

We refer the reader to Section \ref{sec:Alternations} for the construction of the alternations relevant to this paper and to Section \ref{sec:Cancellation} for further applications.

\subsection{Generic Sylow theory}
\label{sec:Order polynomial}

In \cite{Bro-Mal92} Brou\'e and Malle introduced an analogue of the theory of $\ell$-subgroups in the context of connected reductive groups. The main idea is to replace prime powers with powers of cyclotomic polynomials. To start, observe that for every connected reductive group $\G$ defined over an algebraically closed field $\mathbb{F}$ of characteristic $p$, with a Frobenius endomorphism $F:\G\to \G$ associated with an $\mathbb{F}_q$-structure on $\G$ ($q$ a power of $p$), we can associate a polynomial $P_{(\G,F)}(x)\in\mathbb{Z}[x]$, known as the \textit{order polynomial} of $(\G,F)$, with the property that $P_{(\G,F)}(q)=|\G^F|$ (see \cite[Definition 1.6.10 and Remark 1.6.15]{Gec-Mal20}). Now, for a set of positive integers $E$, we say that an $F$-stable torus $\S$ of $\G$ is a \textit{$\Phi_E$-torus} of $(\G,F)$ if its order polynomial $P_{(\S,F)}(x)$ is the product of $e$-th cyclotomic polynomials $\Phi_e(x)$ with $e\in E$. Centralisers of $\Phi_E$-tori are called \textit{$E$-split Levi subgroups} of $(\G,F)$. When considering the set $E=\{e\}$, we omit parentheses and talk about $\Phi_e$-tori, $e$-split Levi subgroups etc. Moreover, recall that every $F$-stable torus $\T$ contains a unique maximal $\Phi_E$-torus denoted by $\T_{\Phi_E}$. We define $\mathcal{T}_E(\G,F)$ to be the set of $\Phi_E$-tori of $(\G,F)$ containing $\z^\circ(\G)_{\Phi_E}$ and consider its subset $\mathcal{T}_E^\star(\G,F)$ consisting of those $\Phi_E$-tori of $(\G,F)$ strictly containing $\z^\circ(\G)_{\Phi_E}$. Similarly, we write $\CL_E(\G,F)$ to denote the set of $E$-split Levi subgroups of $(\G,F)$ while the subset of proper (that is, strictly contained in $\G$) $E$-split Levi subgroups of $(\G,F)$ is denoted by $\CL_E^\star(\G,F)$. These sets have a natural structure of $\G^F$-posets with respect to the conjugation action of $\G^F$ and the usual inclusion of subgroups. As explained in Section \ref{sec:Notations chains}, we can then form the corresponding simplicial complexes. By convention, we assume that the simplices of $\Delta(\mathcal{T}_E(\G,F))$ always have $\z^\circ(\G)_{\Phi_E}$ (the minimum of the poset $\mathcal{T}_E(\G,F)$) as a starting term. Similarly, we assume that the simplices of $\Delta(\CL_e(\G,F))$ have $\G$ (the maximum of the poset $\CL_E(\G,F)$) as a final term.

We define $e_{\ell}(q)$ to be the order of $q$ modulo $\ell$ (modulo $4$ if $\ell=2$) and $E_{\ell}(q):=\{n\in\mathbb{N}\mid n_{\ell'}=e_{\ell}(q)\}$. The latter coincides with the set of positive integers $n$ such that $\ell$ divides $\Phi_n(q)$. Notice that in this paper, we will often write $A_\ell:=\O_\ell(A)$ whenever $A$ is an abelian finite group.

\begin{lem}
\label{lem:Abelian subgroup of e-split Levi subgroup}
Assume that $\ell$ is good for $\G$ and does not divide $|\z(\G)^F:\z^\circ(\G)^F|$ and consider a subset $E\subseteq E_{\ell}(q)$. If $\L$ is an $E$-split Levi subgroup of $(\G,F)$, then $\L=\c_\G(\z^\circ(\c_\G^\circ(\z(\L)^F_\ell))_{\Phi_E})$.
\end{lem}

\begin{proof}
First, under our assumption, \cite[Proposition 13.19]{Cab-Eng04} implies that $\L=\c^\circ_\G(\z(\L)^F_\ell)$. On the other hand, $\L=\c_\G(\z^\circ(\L)_{\Phi_E})$ (see, for instance, \cite[Proposition 3.5.5]{Gec-Mal20} or \cite[Lemma 2.5 (i)]{Ros-Generalized_HC_theory_for_Dade}). Now, the result follows by combining the above two equalities.
\end{proof}

\subsection{Centralisers of abelian $\ell$-subgroups}
\label{sec:Centralisers of abelian subgroups}

In this section we collect some well-known results on centralisers of abelian $\ell$-subgroups (see \cite[Section 2]{Cab-Eng94} and \cite[Section 3]{Cab-Eng99}). Remember that $\ell$ is a fixed prime different from the defining characteristic $p$ of $\G$.

\begin{lem}
\label{lem:Abelian subgroups to Levi subgroups}
Let $A$ be an abelian $\ell$-subgroup of $\G^F$ and assume that $\ell$ is good for $\G$. Then $\c_\G^\circ(A)$ is an $F$-stable Levi subgroup of $\G$. If in addition $\ell$ does not divide $|\z(\G)^F:\z^\circ(\G)^F|$ nor $|\z(\G^*)^F:\z^\circ(\G^*)^F|$, then:
\begin{enumerate}
\item $\c_\G^\circ(A)^F=\c_\G(A)^F$;
\item $A\leq \z^\circ(\c_\G^\circ(A))$; and
\item $\c_\G^\circ(A)$ is an $E_{\ell}(q)$-split Levi subgroup.
\end{enumerate}
\end{lem}

\begin{proof}
The connected centraliser $\c_\G^\circ(A)$ is a Levi subgroup according to \cite[Proposition 2.1 (ii)]{Cab-Eng94}. Assume then that $\ell$ does not divide $|\z(\G)^F:\z^\circ(\G)^F|$ nor $|\z(\G^*)^F:\z^\circ(\G^*)^F|$. Under this assumption, property (i) follows by the argument used to prove \cite[Proposition 2.1 (iii)]{Cab-Eng94}, while the properties (ii) and (iii) follow immediately from (i) as explained in the proof of \cite[Lemma 2.6 (ii)-(iii)]{Ros-Generalized_HC_theory_for_Dade}.
\end{proof}

\begin{cor}
\label{cor:Abelian subgroups to e-split Levi subgroups}
Let $A$ be an abelian $\ell$-subgroup of $\G^F$ and assume that $\ell$ is good for $\G$. Then:
\begin{enumerate}
\item $\c_\G(\z^\circ(\c_\G^\circ(A))_{\Phi_e})$ is an $e$-split Levi subgroup of $(\G,F)$;
\item if in addition $\ell$ does not divide $|\z(\G)^F:\z^\circ(\G)^F|$ nor $|\z(\G^*)^F:\z^\circ(\G^*)^F|$, then $A\leq \c_\G(\z^\circ(\c_\G^\circ(A))_{\Phi_e})$.
\end{enumerate}
\end{cor}

\begin{proof}
Set $\L:=\c_\G^\circ(A)$ and notice that this is an $F$-stable Levi subgroup by Lemma \ref{lem:Abelian subgroups to Levi subgroups}. Then $\z^\circ(\L)$ is an $F$-stable torus and by \cite[Theorem 3.4 (2)]{Bro-Mal92} it contains a unique maximal $\Phi_e$-torus $\z^\circ(\L)_{\Phi_e}$. Then $\c_\G(\z^\circ(\L)_{\Phi_e})$ is an $e$-split Levi subgroup of $(\G,F)$. Suppose now that $\ell$ does not divide $|\z(\G)^F:\z^\circ(\G)^F|$ nor $|\z(\G^*)^F:\z^\circ(\G^*)^F|$. By Lemma \ref{lem:Abelian subgroups to Levi subgroups} we get $A\leq \z^\circ(\c_\G^\circ(A))$ and hence $A$ centralises $\z^\circ(\c_\G^\circ(A))$. From this, we deduce that $A\leq \c_\G(\z^\circ(\c_\G^\circ(A)))\leq \c_\G(\z^\circ(\c_\G^\circ(A))_{\Phi_e})$ as required. 
\end{proof}

\section{Homotopy equivalences and applications}
\label{sec:Homotopy equivalence}

Let $\G$ be a connected reductive group defined over an algebraically closed field $\mathbb{F}$ of characteristic $p$, $F:\G\to \G$ a Frobenius endomorphism defining an $\mathbb{F}_q$-structure on the variety $\G$, and $\ell$ a prime number. The (homotopy type of the) Brown complex $\Delta(\mathcal{S}_\ell^\star(\G^F))$ was described by Quillen in \cite[Theorem 3.1]{Qui78} for the case $\ell=p$. In particular, it was shown that $\Delta(\mathcal{S}_\ell^\star(\G^F))$ is homotopy equivalent to the Tits (spherical) building associated to the set of proper rational parabolic subgroups. In this section we consider the non-defining characteristic case $\ell\neq p$ and show, under suitable hypotheses, that the homotopy type of the Brown complex is controlled by the generic Sylow theory developed by Brou\'e and Malle \cite{Bro-Mal92}. We then consider applications of this result to generic Sylow theory, homology decompositions, and Alperin's Weight Conjecture.

\subsection{Equivariant homotopy equivalence via $e_{\ell}(q)$-split Levi subgroups}

We state a condition on primes that was first considered in the work of Cabanes and Enguehard (see \cite[Condition 22.1]{Cab-Eng04}). This will ensures that non-central abelian $\ell$-subgroups correspond to proper $e_{\ell}(q)$-split Levi subgroups under the map of posets consider in the proof below. Recall that every connected reductive group $\G$ gives rise to a simply connected group $\G_{\rm sc}:=([\G,\G])_{\rm sc}$ as defined in \cite[Example 1.5.3 (b)]{Gec-Mal20}. We freely use the notion of dual group $(\G^*,F^*)$ as in \cite[Definition 1.5.17]{Gec-Mal20}.

\begin{defin}
\label{def:Prime condition}
Let $\pi(\G,F)$ be the set of primes $\ell$ that are good for $\G$, do not divide $2$, $q$ or $|\z(\G_{\rm sc})^F|$, and satisfy $\ell\neq 3$ whenever $(\G,F)$ has a rational component of type ${^3{{\bf D}}_4}$. Then, we define $\pi(\G,F)$ to be the set of primes $\ell\in\pi'(\G,F)$ not diving $|\z(\G)^F:\z^\circ(\G)^F|$ nor $|\z(\G^*)^F:\z^\circ(\G^*)^F|$.
\end{defin}

We now prove Theorem \ref{thm:Main homotopy equivalence for ell-subgroups and e-split Levi}.

%{\color{blue}{
%We could try to relax the hypothesis of the theorem as follows, by considering separately the single case of the triality group. Assume that $\ell$ is odd and good for $\G$ with $\ell$ not dividing the order of $\z(\G_{\rm sc})^F$ nor of $\z(\G)^F$. Then we have the homotopy equivalence.

%PROOF: Reduce to $\G=\G_{\rm ad}$, then reduce to $\G$ rationally irreducible, handle the case $\G^F={^3{\bf{D}}_4}(q)$ and $\ell=3$ separately.
%}}

\begin{theo}
\label{thm:Homotopy equivalence for e-split Levi}
Suppose that $\ell\in\pi(\G,F)$ does not divide the order of $\z(\G)^F$. Then
\[\Delta\left(\mathcal{S}^\star_\ell\left(\G^F\right)\right)\simeq_{\G^F}\Delta\left(\CL_{e_{\ell}(q)}^\star\left(\G,F\right)\right).\]
\end{theo}

\begin{proof}

Set $e:=e_{\ell}(q)$ and recall that $\ab_\ell^\star(\G^F)$ is the poset of non-trivial abelian $\ell$-subgroups of $\G^F$. Since $\ab_\ell^\star(\G^F)$ contains the poset of non-trivial elementary abelian $\ell$-subgroups considered by Quillen in \cite{Qui78}, we deduce from \cite[Theorem 2]{The-Web91} that the Brown complex $\Delta(\mathcal{S}_\ell^\star(\G^F))$ is $\G^F$-homotopy equivalent to $\Delta(\ab_\ell^\star(\G^F))$. On the other hand, since opposite posets are associated to homeomorphic simplicial complexes, there exists a $\G^F$-equivariant homeomorphism between $\Delta(\CL_e^\star(\G,F))$ and $\Delta(\CL_e^\star(\G,F)^{\rm op})$ and hence it is enough to show that the simplicial complexes $\Delta(\ab_\ell^\star(\G^F))$ and $\Delta((\CL_e^\star(\G,F))^{\rm op})$ are $\G^F$-homotopy equivalent. To prove the latter statement, we apply the results of Section \ref{sec:Homotopy} and Section \ref{sec:Centralisers of abelian subgroups}. Consider the map
\begin{align*}
\phi_e:\ab_\ell^\star\left(\G^F\right)&\to\CL_e^\star\left(\G,F\right) ^{\rm op}
\\
A &\mapsto \c_\G\left(\z^\circ\left(\c_\G^\circ(A)\right)_{\Phi_e}\right)\nonumber
\end{align*}
and observe that this is a well-defined map of $\G^F$-posets. In fact, according to Corollary \ref{cor:Abelian subgroups to e-split Levi subgroups} we know that $\phi_e(A)$ is an $e$-split Levi subgroup and it is enough to show that $\phi_e(A)$ is strictly contained in $\G$. To prove this fact, we first apply Lemma \ref{lem:Abelian subgroups to Levi subgroups} (iii) to deduce that $\H:=\c_\G^\circ(A)$ is an $E_{\ell}(q)$-split Levi subgroup of $(\G,F)$. Moreover, recalling that $\z(\G)^F_\ell=1$, we observe that $A$ cannot be central in $\G$ and therefore that $\H<\G$. If now we assume $\phi_e(A)=\G$, it follows that $\z^\circ(\H)_{\Phi_e}\leq \z(\G)$ and \cite[Lemma 22.3]{Cab-Eng04} (see also the proof of \cite[Theorem 22.2]{Cab-Eng04}) implies that $\z^\circ(\H)_{\Phi_{e\ell^a}}\leq \z(\G)$ for every $a\geq 0$. But then $\z^\circ(\H)_{\Phi_{E_{\ell}(q)}}\leq \z(\G)$ and, since $\H$ is an $E_{\ell}(q)$-split Levi subgroup of $(\G,F)$, we obtain $\H=\G$. This contradiction shows that $\phi_e(A)$ is strictly contained in $\G$ and hence the map $\phi_e$ is well-defined. Next, we show that $\phi_e$ is a map of $\G^F$-posets. By the definition of $\phi_e$ it follows that $\phi_e(A^g)=\phi_e(A)^g$ for every abelian $\ell$-subgroup $A\leq\G^F$ and $g\in\G^F$. Moreover, let $A\leq A'$ be two abelian $\ell$-subgroups and set $\L:=\c_\G^\circ(A)$ and $\K:=\c_\G^\circ(A')$. Since $\K\leq \L$ are Levi subgroups of $(\G,F)$ (see Lemma \ref{lem:Abelian subgroups to Levi subgroups}), we obtain $\z(\L)\leq \z(\K)$ which implies $\z^\circ(\L)_{\Phi_e}\leq\z^\circ(\K)_{\Phi_e}$ and therefore $\phi_e(A')\leq \phi_e(A)$. This shows that $\phi$ is a map of $\G^F$-posets.

We now want to apply Lemma \ref{lem:Criterion for G-homotopy} to show that $\phi_e$ induces a $\G^F$-homotopy equivalence $\Delta(\phi_e)$. For this purpose, fix an $e$-split Levi subgroup $\L$ of $(\G,F)$ and observe that the elements of the fibre $\mathcal{X}_\L:=\phi_e^{-1}(\CL_e^\star(\G,F)^{\rm op}_{\geq \L})$ are those abelian $\ell$-subgroups $A$ of $\G^F$ that satisfy $\phi_e(A)\leq \L$. We need to prove that $\mathcal{X}_\L$ is $\n_\G(\L)^F$-contractible and we do so by showing that $\mathcal{X}_\L$ is $\n_\G(\L)^F$-join-contractible via $A_0:=\z(\L)^F_\ell$ (see Corollary \ref{cor:Criterion for G-join-contractibility}). First, observe that $A_0$ is an element of the fibre $\mathcal{X}_\L$, since $\phi_e(A_0)=\L$ according to Lemma \ref{lem:Abelian subgroup of e-split Levi subgroup}, and that $A_0$ is fixed by the action of $\n_\G(\L)^F$. Moreover, whenever $A\in\mathcal{X}_\L$, Corollary \ref{cor:Abelian subgroups to e-split Levi subgroups} implies that $A\leq \phi_e(A)\leq \L$ and therefore $[A,A_0]=1$ since $A_0$ centralises $\L$. It follows that the product $AA_0$ is a well-defined abelian $\ell$-subgroup of $\G^F$ that satisfies $\phi_e(AA_0)\leq \L$. This shows that $AA_0$, which is the join of $A$ and $A_0$, is defined in the poset $\mathcal{X}_\L$ for every $A\in\mathcal{X}_\L$ and the proof is now complete.
\end{proof}

The homotopy equivalence constructed above can also be stated in terms of $\Phi_{e_{\ell}(q)}$-tori. In fact, it is not hard to show that, for every set of positive integers $E$, the simplicial complex of $E$-split Levi subgroups of $(\G,F)$ is $\G^F$-homotopy equivalent to that of $\Phi_E$-tori of $(\G,F)$. Recall that the notation $\mathcal{T}^\star_E(\G,F)$ is used to denote the set of $\Phi_E$-tori $\S$ of $(\G,F)$ that properly contain $\z(\G)_{\Phi_E}$.

\begin{prop}
\label{prop:Homotopy equivalence for e-tori}
For every set of positive integers $E$ we have
\[\Delta\left(\CL_E^\star\left(\G,F\right)\right)\simeq_{\G^F}\Delta\left(\mathcal{T}_E^\star\left(\G,F\right)\right).\]
\end{prop}

\begin{proof}
Arguing as in the proof of Theorem \ref{thm:Homotopy equivalence for e-split Levi}, it is enough to show that the simplicial complex $\Delta(\mathcal{T}_E^\star(\G,F))$ is $\G^F$-homotopy equivalent to $\Delta(\CL_E^\star(\G,F)^{\rm op})$. To construct this equivalence, we consider the map of $\G^F$-poset $\phi:\mathcal{T}_E^\star(\G,F)\to\CL_E^\star(\G,F)^{\rm op}$ that sends a $\Phi_E$-torus $\S$ to its centraliser $\c_\G(\S)$. Observe that $\c_\G(\S)$ is an $E$-split Levi subgroup of $(\G,F)$ and is a proper subgroup of $\G$ because $\z(\G)_{\Phi_E}$ is properly contained in $\S$. Now let us fix an $E$-split Levi subgroup $\L$ of $(\G,F)$, set $\S_0:=\z^\circ(\L)_{\Phi_E}$ and denote by $\mathcal{X}_\L$ the fibre $\phi^{-1}(\CL_E^\star(\G,F)_{\geq \L}^{\rm op})$. If $\S\in\mathcal{X}_\L$, then $\S\leq \c_\G(\S)\leq \L$ and $[\S,\S_0]=1$ since $\S_0$ is central in $\L$. In particular, $\S\S_0$ is a $\Phi_E$-torus of $(\G,F)$ satisfying $\phi(\S\S_0)\leq \L$ and it follows from Corollary \ref{cor:Criterion for G-join-contractibility} that $\mathcal{X}_\L$ is $\n_\G(\L)^F$-contractible. Now the result follows from Lemma \ref{lem:Criterion for G-homotopy}.
\end{proof}

%{\color{blue}{
%\begin{rmk}[Not sure about this one]
%Observe that a Levi subgroup $\L$ is $1$-split if and only if it is \textit{maximally split}, i.e. it is contained in an $F$-stable parabolic subgroup. Then it is natural to ask whether
%\[\mathcal{T}\left(\G^F\right)\simeq_{\G^F}\Delta\left(\CL_1\left(\G,F\right)\right).\]
%This fact together with \eqref{eq:Buildings} and \eqref{eq:Homotopy equivalence} would imply that
%\[\Delta\left(\mathcal{S}_p\left(\G^F\right)\right)\simeq_{\G^F}\Delta\left(\mathcal{S}_\ell\left(\G^F\right)\right)\]
%whenever $o(q)\equiv 1\pmod{\ell}$. This provides an explicit relation between the local structure of $\G^F$ in the defining characteristic and certain non-defining characteristics (recall Bonnaf\'e's presentation and Rouquier's paper on the connecteion between defining VS non-defining!) 
%\end{rmk}
%}}

\subsection{Homological generic Sylow theorem}

The Brown complex of a finite group was first studied by K. S. Brown in \cite{Bro75} where he proved a homological version of the third Sylow theorem. More precisely, he showed that for every prime number $\ell$ dividing the order of a finite group $G$, the Euler characteristic $\chi(\Delta(\mathcal{S}_\ell^\star(G)))$ is congruent to $1$ modulo the order of a Sylow $\ell$-subgroup. On the other hand, Boru\'e and Malle developed a generic version of the Sylow theorems for finite reductive groups. In this section, by exploiting the homotopy equivalence constructed in Theorem \ref{thm:Homotopy equivalence for e-split Levi}, we obtain an analogous congruence involving the euler characteristic of the simplicial complex of $e$-split Levi subgroups when considering large primes. This gives evidence for a homological generic Sylow theorem.

We say that a prime $\ell$ is \textit{large} for $(\G,F)$ if there exists a unique positive integer $e$ such that $\Phi_{e}(x)$ divides the order polynomial $P_{(\G,F)}(x)$ and $\ell$ divides $\Phi_{e}(q)$ (see \cite[Definition 5.1]{Bro-Mal-Mic93}). In this case, we also say that $\ell$ is \textit{$(\G, F, e)$-adapted} (see \cite[Definition 5.3]{Bro-Mal-Mic93}). Observe that if $\ell$ is large for $(\G,F)$, then $\ell$ is good for $\G$ and does not divide $|\z(\G)^F:\z^\circ(\G)^F|$ nor $|\z(\G^*)^F:\z^\circ(\G^*)^F|$ according to \cite[Proposition 5.2]{Bro-Mal-Mic93}. Before proving our next result, we point out that Theorem \ref{thm:Homotopy equivalence for e-split Levi} holds for large primes $\ell$ not dividing the order of $\z(\G)^F$.

\begin{rmk}
\label{rmk:Large primes}
In the proof of Theorem \ref{thm:Homotopy equivalence for e-split Levi} we use the hypothesis $\ell\in\pi'(\G,F)$ only when invoking \cite[Lemma 22.3]{Cab-Eng04}. This lemma is used to show that if $\H$ is an $F$-stable Levi subgroup of $\G$ satisfying $\z^\circ(\H)_{\Phi_e}\leq \z(\G)$, then we get $\z^\circ(\H)_{\Phi_{e\ell^a}}\leq \z(\G)$ for every $a\geq 0$. However, if $\ell$ is large for $(\G,F)$ and $(\G,F,e)$-adapted, then we get $\z^\circ(\H)_{\Phi_{e}}=\z^\circ(\H)_{\Phi_{e\ell^a}}$ for every $a \geq 0$ and the implication above trivially holds.
\end{rmk}

\begin{cor}
\label{cor:Homological Sylow and e-split Levi}
Suppose that $\ell$ is large for $(\G,F)$ and does not divide the order of $\z(\G)^F$. Then
\[\chi\left(\Delta\left(\CL_{e_{\ell}(q)}^\star\left(\G,F\right)\right)\right)\equiv 1 \enspace\left({\rm mod}\enspace\left(\Phi_{e_{\ell}(q)}(q)^a\right)_\ell\right)\]
where $a$ is the $\Phi_{e_{\ell}(q)}$-valuation of the order polynomial of $(\G,F)$.
\end{cor}

\begin{proof}
Set $e:=e_{\ell}(q)$. We may assume without loss of generality that our prime $\ell$ is $(\G,F,e)$-adapted. In fact, since $\ell$ divides $\Phi_e(q)$, if $\ell$ is $(\G,F,e_0)$-adapted for some positive integer $e_0$ then either $e=e_0$ or $\Phi_e$ does not divide the order polynomial of $(\G,F)$. In the latter case, we conclude that $a(e)=0$ and there are no non-trivial $\Phi_e$-subgroups of $(\G,F)$. Thus $\chi(\mathcal{T}_e(\G,F))=0$, $\Phi^{a(e)}_e(q)_\ell=1$ and the result follows trivially. Next, observe that the Euler characteristic $\chi(\Delta)$ of a simplicial complex $\Delta$ only depends on the homology of $\Delta$ and hence it is a homotopy invariant. Therefore, applying Theorem \ref{thm:Homotopy equivalence for e-split Levi} (which holds under our assumption thanks to Remark \ref{rmk:Large primes}) we get
\begin{equation}
\label{eq:Homological Sylow and e-split Levi, 1}
\chi\left(\Delta\left(\CL_e^\star(\G,F)\right)\right)=\chi\left(\Delta\left(\mathcal{S}_\ell^\star\left(\G^F\right)\right)\right).
\end{equation}
On the other hand, we know from \cite[Corollary 3.13 (ii)]{Bro-Mal92} that
\begin{equation}
\label{eq:Homological Sylow and e-split Levi, 2}
\left|\G^F\right|_\ell=\Phi_e^{a}(q)_\ell
\end{equation}
where $a$ is the $\Phi_{e_{\ell}(q)}$-valuation of the order polynomial of $(\G,F)$. The result now follows from \eqref{eq:Homological Sylow and e-split Levi, 1} and \eqref{eq:Homological Sylow and e-split Levi, 2} as a consequence of Brown's theorem \cite[Corollary 2]{Bro75}.
\end{proof}

Observe that by using Proposition \ref{prop:Homotopy equivalence for e-tori} we immediately deduce that the congruence of the above corollary could also be stated in terms of the simplicial complex of $\Phi_{e_{\ell}(q)}$-tori as
\[\chi\left(\Delta\left(\mathcal{T}_{e_{\ell}(q)}^\star\left(\G,F\right)\right)\right)\equiv 1 \enspace\left({\rm mod}\enspace\left(\Phi_{e_{\ell}(q)}(q)^a\right)_\ell\right).\]
We conclude this section with an alternative proof of the above corollary which involves the reduced Lefschetz module (an analogue of the Steinberg module) and is basically due to Quillen.

\begin{rmk}
The above congruence can also be obtained as a consequence of Quillen's result on the projectivity of the generalised Steinberg module. For this purpose, set $e=e_{\ell}(q)$ and define the \textit{reduced Lefschetz module} (see \cite[Definition 6.3.2]{Ben98}) of $\Delta(\CL_e^\star(\G,F))$ over an algebraically closed field $\overline{\mathbb{F}}_\ell$ of characteristic $\ell$
\[{\rm St}_e(\G,F):=\tilde{L}_{\G^F}\left(\Delta\left(\CL_e^\star(\G,F)\right),\overline{\mathbb{F}}_\ell\right)\]
whose dimension is $\chi(\Delta(\CL_e^\star(\G,F)))-1$. If we now assume that $\ell$ is large for $(\G,F)$ and does not divide the order of $\z(\G)^F$, then Theorem \ref{thm:Homotopy equivalence for e-split Levi} (together with Remark \ref{rmk:Large primes}) implies that ${\rm St}_e(\G,F)$ coincides with the \textit{generalised Steinberg module} ${\rm St}_\ell(\G^F)$ as defined in \cite[Definition 6.7.1]{Ben98}. On the other hand, by applying \cite[Corollary 4.3]{Qui78} we deduce that ${\rm St}_\ell(\G^F)$ is a virtual projective module and hence its dimension is a multiple of $|\G^F|_\ell$. This shows that that the order of a Sylow $\ell$-subgroup of $\G^F$ divides $\chi(\CL_e^\star(\G,F))-1$. Corollary \ref{cor:Homological Sylow and e-split Levi} now follows from \cite[Corollary 3.13 (ii)]{Bro-Mal92} which shows that \eqref{eq:Homological Sylow and e-split Levi, 2} holds.
\end{rmk}

\subsection{Sharp homology decompositions}
\label{sec:Homology decompositions}

The idea of decomposing the classifying space of a finite group $G$, up to mod $\ell$ homology, in terms of classifying spaces of subgroups of $G$ was systematically studied by Dwyer in \cite{Dwy97} building on previous works of Brown \cite{Bro79}, Webb \cite{Web87}, Jackowski--McClure \cite{JM92}, and Jackowski--McClure--Oliver \cite{JMO92}. The exact definition of a \textit{homology decomposition} can be found in \cite{Dwy97} and is related to the notion of \textit{ampleness} (see \cite[Definition 1.2]{Dwy97}). Here, we consider an extension of Dwyer's definition to more general simplicial complexes.

\begin{defin}[{{\cite[Definition 5.6.6]{Ben-Smi08}}}]
Let $G$ be a finite group and $\Delta$ a $G$-simplicial complex. Denote by $\star$ the one point space with trivial $G$-action and by $(-)_{\rm hG}$ the Borel construction \cite[Definition 2.5.9]{Ben-Smi08}. We say that $\Delta$ is \textit{ample} (with respect to the prime $\ell$) if the map
\[(\Delta)_{\rm hG}\to(\star)_{\rm hG}\]
given by $\Delta\to \star$ induces an isomorphism on mod $\ell$ homology.
\end{defin}

For every ample $G$-simplicial complex we get three homology decompositions: the \textit{centralizer decomposition}, the \textit{subgroup decomposition}, and the \textit{normalizer decomposition} (see \cite[Theorems 1.4, 1.6, 1.8]{Dwy97} or \cite[Theorem 5.6.10]{Ben-Smi08}). Moreover, observe that the notion of ampleness is a homotopy invariant of the simplicial complex $\Delta$ (see, for instance, \cite[Lemma 5.6.8]{Ben-Smi08}). Since the Brown complex $\Delta(\mathcal{S}_\ell^\star(G))$ is ample, with respect to the prime $\ell$, according to \cite[Theorem 5.1.3 and Proposition 5.6.7]{Ben-Smi08}, we get the following consequence of Theorem \ref{thm:Homotopy equivalence for e-split Levi}.

\begin{cor}
\label{cor:Ampleness}
Suppose that $\ell\in\pi(\G,F)$ does not divide the order of $\z(\G)^F$. Then the $\G^F$-simplicial complex $\Delta(\CL_{e_{\ell}(q)}^\star(\G,F))$ is ample with respect to the prime $\ell$.
\end{cor}

Next, recall that a homology decomposition is called \textit{sharp} if the corresponding Bousfield--Kan cohomology spectral sequence of the homotopy colimit satisfies $E^2_{i,j}=0$ for every $i>0$ and every $j$ (see \cite{Dwy98} for further details). A simplicial complex $\Delta$ is called \textit{centraliser sharp}, \textit{subgroup sharp}, or \textit{normalizer sharp} (with respect to $\ell$) if it is ample (with respect to $\ell$) and the corresponding homology decomposition is sharp in the sense described above. This notion was introduced by Dwyer \cite{Dwy98} and further studied by Grodal \cite{Gro02} and Grodal--Smith \cite{Gro-Smi06}. Unlike the notion of ampleness, being centraliser sharp or subgroup sharp is not a homotopy invariant. However, according to \cite[Corollary 7.2]{Gro02} (see also \cite[Theorem 5.8.11]{Ben-Smi08}) this is true in the case of normaliser sharpness. Since the Brown complex $\Delta(\mathcal{S}_\ell^\star(G))$ is normalizer sharp with respect to the primes $\ell$ (see, for instance, \cite[Theorem 5.8.8]{Ben-Smi08}) our homotopy equivalence from Theorem \ref{thm:Homotopy equivalence for e-split Levi} implies the following.

\begin{cor}
\label{cor:Sharpness}
Suppose that $\ell\in\pi(\G,F)$ does not divide the order of $\z(\G)^F$. Then, the $\G^F$-simplicial complex $\Delta(\CL_{e_{\ell}(q)}^\star(\G,F))$ is normaliser sharp with respect to the primes $\ell$.
\end{cor}

The above result leads to further natural questions. For instance, whether the above simplicial complex is centraliser sharp or subgroup sharp. Moreover, it would be interesting to know if the restrictions on the prime $\ell$ can be relaxed. Finally, we mention that using recent results of Grodal \cite{Gro23}, the $\G^F$-homotopy equivalence of Theorem \ref{thm:Homotopy equivalence for e-split Levi} provides a way to study endotrivial modules for finite reductive groups (in non-defining characteristic $\ell\in\pi(\G,F)$) in terms of $e_{\ell}(q)$-split Levi subgroups (or,equivalently, of $\Phi_{e_{\ell}(q)}$-tori according to Proposition \ref{prop:Homotopy equivalence for e-tori}).

\subsection{Alperin's Weight Conjecture}
\label{sec:AWC}

In \cite{Alp87}, inspired by the weight theory for the modular representations of finite reductive groups in defining characteristic, Alperin introduced his celebrated Weight Conjecture. If $G$ is a finite group and $\ell$ a prime, we call a pair $(Q,\vartheta)$ an \textit{$\ell$-weight} of $G$ if $Q$ is an $\ell$-subgroup of $G$ and $\vartheta$ is an irreducible character of $\n_G(Q)/Q$ of $\ell$-defect zero, that is $\vartheta(1)_\ell=|\n_G(Q):Q|_\ell$.

\begin{alp}
\label{conj:Alperin}
The number of irreducible $\ell$-Brauer characters of $G$ equals the number of conjugacy classes of $\ell$-weights of $G$.
\end{alp}

A major breakthrough towards the understanding of the above statement was made by Kn\"orr and Robinson in \cite{Kno-Rob89} where they presented a reformulation in terms of the Brown complex of $\ell$-subgroups. In particular, this allows one to show that such a conjecture only depends on the homotopy type of the Brown complex. This fact is even more explicit in the treatment of the Kn\"orr--Robinson reformulation given by Webb \cite[Section 6]{Web87}. This can be stated by saying that
\begin{equation}
\label{eq:Alperin Webb}
\k^0(G)=-l_\ell\left(\tilde{\Lambda}(\Delta(\mathcal{S}_\ell^\star(G)))\right)
\end{equation}
where $\k^0(G)$ denotes the number of irreducible characters of $G$ of $\ell$-defect zero and $\tilde{\Lambda}$ is the \textit{reduced Lefschetz invariant} in the Burnside ring $b(G)$ \cite[Definition 6.3.3]{Ben98}. Here, $l_\ell(H)$ denotes the number of irreducible $\ell$-Brauer characters of a finite group $H$ and we regard $l_\ell$ as a function on the Burnside ring $b(G)$ (see \cite[Section 6]{Web87} and \cite[Section 6.9]{Ben98} for further details). Since $\tilde{\Lambda}$ is a homotopy invariant, the term on the right-hand side of \eqref{eq:Alperin Webb} can be equivalently stated in terms of various subgroup complexes. In particular, we can restate Alperin's Weight Conjecture for finite reductive groups in terms of the simplicial complex of $e_{\ell}(q)$-split Levi subgroups.

\begin{cor}
\label{cor:Alperin Webb}
Suppose that $\ell\in\pi(\G,F)$ does not divide the order of $\z(\G)^F$. Then the Kn\"orr--Robinson reformulation of Alperin's Weight Conjecture (in the sense of Webb \eqref{eq:Alperin Webb}) holds for $\G^F$ at the prime $\ell$ if and only if
\[\k^0\left(\G^F\right)=-l_\ell\left(\tilde{\Lambda}\left(\Delta\left(\CL_{e_{\ell}(q)}^\star(\G,F)\right)\right)\right).\]
\end{cor}

Another approach to Alperin's Weight Conjecture was introduced by Th\'evenaz \cite{The93} in relation to the equivariant $K$-theory of the Brown complex. For every finite group $G$ and $G$-simplicial complex $\Delta$, define the \textit{equivariant Euler characteristic} of $\Delta$ as
\[\chi_G(\Delta):=\dim\left(\mathbb{Q}\otimes K_G^0(\Delta)\right)-\dim\left(\mathbb{Q}\otimes K_G^1(\Delta)\right)\]
where $K_G^0$ and $K_G^1$ are the equivariant $K$-theory groups and $\Delta$ is interpreted as a compact $G$-space (see \cite{The93} for further details). Th\'evenaz proved that the Kn\"orr--Robinson reformulation of Alperin's Weight Conjecture can be restated as
\begin{equation}
\label{eq:Alperin Thevenaz}
\k\left(G\right)-\k^0\left(G\right)=\chi_G\left(\Delta\left(\mathcal{S}_\ell^\star(G)\right)\right)
\end{equation}
where $\k(G)$ denotes the number of irreducible characters of $G$. Once again, observe that the right-hand side of the above equation is invariant under equivariant homotopy equivalences. Since the homotopy equivalence constructed in Theorem \ref{thm:Homotopy equivalence for e-split Levi} is a $\G^F$-homotopy equivalence, we can restate the right-hand side of \eqref{eq:Alperin Thevenaz} for finite reductive groups by using the simplicial complex of $e_{\ell}(q)$-split Levi subgroups.

\begin{cor}
\label{cor:Alperin Thevenaz}
Suppose that $\ell\in\pi(\G,F)$ does not divide the order of $\z(\G)^F$. Then the Kn\"orr--Robinson reformulation of Alperin's Weight Conjecture (in the sense of Webb \eqref{eq:Alperin Thevenaz}) holds for $\G^F$ at the prime $\ell$ if and only if
\[\k\left(\G^F\right)-\k^0\left(\G^F\right)=\chi_{\G^F}\left(\Delta\left(\CL_{e_{\ell}(q)}^\star(\G,F)\right)\right).\]
\end{cor}

We point out that while Alperin's Weight Conjecture is not equivalent to the Kn\"orr--Robinson reformulation for every fixed finite group, the latter is equivalent to \eqref{eq:Alperin Webb} and \eqref{eq:Alperin Thevenaz} for every fixed finite group according to \cite[Corollary 4.5]{Kno-Rob89} (see also \cite[Corollary 9.23]{Nav18})

Th\'evenaz approach has recently been considered by M{\o}ller to obtain new proofs of Alperin's Weight Conjecture (in the Kn\"orr--Robinson reformulation) for general linear, unitary and symplectic groups with respect to the defining characteristic \cite{Mol19}, \cite{Mol21} and \cite{Mol22}. Corollary \ref{cor:Alperin Thevenaz} together with the geometric nature of the $e$-split Levi subgroups might lead to further developments in this direction for the non-defining characteristics.

To conclude this section, we point out that using the homotopy equivalence constructed in Theorem \ref{thm:Homotopy equivalence for e-split Levi} is not enough to obtain a reformulation for the blockwise version of Alperin's Weight Conjecture in terms of $e_{\ell}(q)$-split Levi subgroups. In fact, the latter does not admit (known) formulations in the form of \eqref{eq:Alperin Webb} and \eqref{eq:Alperin Thevenaz}. Nevertheless, in the following sections we will extend the arguments that led to the proof of Theorem \ref{thm:Homotopy equivalence for e-split Levi} and obtain results that are compatible with block theory and with even stronger conjectures.

\section{Adapting abelian $\ell$-subgroups to $e$-split Levi subgroups}
\label{sec:Adapting ell-groups to e-split Levi}

In the previous section we have constructed a $\G^F$-homotopy equivalence between the Brown complex $\Delta(\mathcal{S}_\ell^\star(\G^F))$ and the simplicial complex $\Delta(\CL_{e_{\ell}(q)}^\star(\G,F))$ whenever $\ell$ is a prime different from the defining characteristic of $\G$ and satisfying suitable hypotheses. This was enough to reformulate Alperin's Weight Conjecture, for the prime $\ell$, in terms of chains belonging to $\Delta(\CL_{e_{\ell}(q)}^\star(\G,F))$. However, in order to obtain a similar reformulation for Dade's Conjecture and its inductive condition, i.e. the Character Triple Conjecture, we need to refine these arguments. More precisely, we identify a distinguished subset of $\ell$-subgroups of $\G^F$ that corresponds to that of $e_{\ell}(q)$-split Levi subgroups of $(\G,F)$ (in a sense specified in Lemma \ref{lem:e-split produces e-closed}) and show that the only chains of $\Delta(\mathcal{S}_\ell(\G^F))$ that contribute to these conjectures are those whose terms lie in this newly defined subset of $\ell$-subgroups. This is done by constructing a series of alternations as introduced in Definition \ref{def:Alternations}. Our approach extends (unpublished) ideas of Broué, Fong and Srinivasan \cite{Bro-Fon-Sri1} that appeared in \cite[Section 7.2]{Ros-Generalized_HC_theory_for_Dade} and were used to settle analogous questions for large primes.

\subsection{The (weak) $e$-closure of an abelian $\ell$-subgroup}

Recall from Section \ref{sec:Notations chains} that $\ab_\ell(\G^F)$ denotes the set of abelian $\ell$-subgroups of $\G^F$. Our first aim is to define the following $e$-closure operator on this class of subgroups. This should be interpreted as an attempt to adapt these subgroups to the class of $e$-split Levi subgroups (see Lemma \ref{lem:e-split produces e-closed} below).

\begin{defin}[$e$-closure]
\label{def:e-closure}
For every positive integer $e$ we define the \textit{$e$-closure} on $\ab_\ell(\G^F)$ as the map
\begin{align*}
\gamma_{\ell,e}:\ab_\ell\left(\G^F\right)&\to\ab_\ell\left(\G^F\right)
\\
A &\mapsto \z^\circ\left(\c_\G\left(\z^\circ\left(\c_\G^\circ(A)\right)_{\Phi_e}\right)\right)^F_\ell
\end{align*}
and we say that an abelian $\ell$-subgroup $A\in\ab_\ell(\G^F)$ is \textit{$e$-closed} if it satisfies $A=\gamma_{\ell,e}(A)$. We denote by $\ab_\ell(\G^F)^{\gamma_{\ell,e}}$ the set of $e$-closed $\ell$-subgroups of $\G^F$, that is, the set of $\gamma_{\ell,e}$-fixed elements of the set $\ab_\ell(\G^F)$.
\end{defin}

If we assume that $\ell$ is good for $\G$, then the connected centraliser of an abelian $\ell$-subgroup is a Levi subgroup according to \cite[Proposition 2.1 (ii)]{Cab-Eng94}. In this case, the map $\gamma_{\ell,e}$ satisfies the following useful properties. Here, we denote by $\aut_\mathbb{F}(\G^F)$ the group of automorphisms of $\G^F$ described in \cite[Section 2.4]{Cab-Spa13}. As usual, recall that $\ell\neq p$.

\begin{lem}
\label{lem:Properties of e-closure}
Let $A$ be an abelian $\ell$-subgroup of $\G^F$ and assume that $\ell$ is good for $\G$. For every positive integer $e$ we have:
\begin{enumerate}
\item $\gamma_{\ell,e}(A)\leq \gamma_{\ell,e}(A')$ for every abelian $\ell$-subgroup $A'\in\ab_\ell(\G^F)$ with $A\leq A'$;
\item $\gamma_{\ell,e}(A^\alpha)=\gamma_{\ell,e}(A)^\alpha$ for every automorphism $\alpha\in\aut_\mathbb{F}(\G^F)$;
\item $\z^\circ(\G)^F_\ell\leq \gamma_{\ell,e}(A)$;
\item $[A,\gamma_{\ell,e}(A)]=1$.
\end{enumerate}
\end{lem}

\begin{proof}
Assuming $A\leq A'$ we get the inclusion $\c_\G^\circ(A')\leq \c_\G^\circ(A)$ and, since $\c_\G^\circ(A')$ is a Levi subgroup of $\c_\G^\circ(A)$, it follows that $\z(\c_\G^\circ(A))\leq \z(\c_\G^\circ(A'))$. Now \cite[Lemma 3.1 (ii)]{Bro-Mal92} yields $\z^\circ(\c_\G^\circ(A))_{\Phi_e}\leq \z^\circ(\c_\G^\circ(A'))_{\Phi_e}$ and we obtain an inclusion of ($e$-split) Levi subgroups $\c_\G(\z^\circ(\c_\G^\circ(A'))_{\Phi_e})\leq \c_\G(\z^\circ(\c_\G^\circ(A))_{\Phi_e})$. As before, taking the centre induces a reverse inclusion $\z^\circ(\c_\G(\z^\circ(\c_\G^\circ(A))_{\Phi_e}))\leq \z^\circ(\c_\G(\z^\circ(\c_\G^\circ(A'))_{\Phi_e}))$ from which we deduce $\gamma_{\ell,e}(A)\leq \gamma_{\ell,e}(A')$. This proves (i) while (ii) follows immediately from the fact that $Y_\ell$ is characteristic in $Y$ and $\S_{\Phi_e}$ is characteristic in $\S$ for every abelian finite group $Y$ and every ($F$-stable) torus $\S$. To prove (iii), observe that $\z^\circ(\G)$ is contained in the $e$-split Levi subgroup $\c_\G(\z^\circ(\c_\G^\circ(A))_{\Phi_e})=:\L$ for every abelian $\ell$-subgroup $A$ of $\G^F$. Then we surely have $\z^\circ(\G)^F_\ell\leq \z^\circ(\L)^F_\ell=\gamma_{\ell,e}(A)$. Finally, since $\c_\G^\circ(A)$ is a Levi subgroup of $\G$ we obtain $\c_\G^\circ(A)=\c_\G(\z^\circ(\c_\G^\circ(A)))\leq \c_\G(\z^\circ(\c_\G^\circ(A))_{\Phi_e})$ (see \cite[Proposition 3.4.6]{Dig-Mic20}) and taking the connected center yields $\z^\circ(\c_\G(\z^\circ(\c_\G^\circ(A))_{\Phi_e}))\leq \c_\G^\circ(A)$. Thus $\gamma_{\ell,e}(A)$ is contained in $\c_\G^\circ(A)\leq \c_\G(A)$ and (iv) follows. 
\end{proof}

Next, we provide a characterisation of $e$-closed abelian $\ell$-subgroups. This also explains why the map $\gamma_{\ell,e}$ is called \textit{$e$-closure} and clarifies the reason for its definition.

\begin{lem}
\label{lem:e-split produces e-closed}
Assume that $\ell$ is good for $\G$ and consider a positive integer $e$. If $A$ is an $e$-closed abelian $\ell$-subgroup $A$ of $\G^F$, then $A=\z^\circ(\L)^F_\ell$ for some $e$-split Levi subgroup $\L$ of $(\G,F)$. The converse holds provide that $e=e_{\ell}(q)$ and $\ell$ does not divide $|\z(\G^F):\z^\circ(\G^F)|$.
\end{lem}

\begin{proof}
Assume that $A$ is an $e$-closed abelian $\ell$-subgroup of $\G^F$. If we set $\L:=\c_\G(\z^\circ(\c_\G^\circ(A))_{\Phi_e})$, then $\L$ is an $e$-split Levi subgroup of $(\G,F)$ and satisfies $A=\gamma_{\ell,e}(A)=\z^\circ(\L)^F_\ell$. Conversely, assume that $\ell$ does not divide $|\z(\G)^F:\z^\circ(\G)^F|$ and that $e=e_{\ell}(q)$. Let $\L$ be an $e$-split Levi subgroup of $(\G,F)$ and set $A:=\z^\circ(\L)^F_\ell$. By Lemma \ref{lem:Abelian subgroup of e-split Levi subgroup} we deduce that $\L=\c_\G(\z^\circ(\c_\G^\circ(A))_{\Phi_e})$ and hence $A=\gamma_{\ell,e}(\z^\circ(\L)^F_\ell)=\gamma_{\ell,e}(A)$ is $e$-closed.
\end{proof}

An immediate consequence of the above lemma is that the connected centraliser of an $e_{\ell}(q)$-closed abelian $\ell$-subgroup is an $e_{\ell}(q)$-split Levi subgroup.

\begin{cor}
\label{cor:e-closed produces e-split}
Assume that $\ell$ is good for $\G$ and does not divide $|\z(\G^F):\z^\circ(\G^F)|$. If $A$ is an $e_{\ell}(q)$-closed abelian $\ell$-subgroup of $\G^F$, then the connected centraliser $\c_\G^\circ(A)$ is an $e_{\ell}(q)$-split Levi subgroup of $(\G,F)$.
\end{cor}

\begin{proof}
By Lemma \ref{lem:e-split produces e-closed} there exists an $e_{\ell}(q)$-split Levi subgroup $\L$ of $(\G,F)$ such that $A=\z^\circ(\L)^F_\ell$. But then $\c_\G^\circ(A)=\c_\G^\circ(\z^\circ(\L)^F_\ell)=\L$ is an $e_{\ell}(q)$-split Levi subgroup of $(\G,F)$ thanks to \cite[Proposition 13.19]{Cab-Eng04}.
\end{proof}

\begin{rmk}
Observe that the converse of Corollary \ref{cor:e-closed produces e-split} does not hold in general. For instance suppose that the prime $\ell$ is good for $\G$, does not divide $|\z(\G)^F:\z^\circ(\G)^F|$ but does divide the order of $\z^\circ(\G)^F$ (consider, for instance, $\G^F=\GL_n(q)$ and $\ell$ dividing $q-1$). Then, for every $\ell$-subgroup $Q$ strictly contained in $\z^\circ(\G)^F_\ell$ we deduce that $\c_\G^\circ(Q)=\G$ is an $e_{\ell}(q)$-split Levi subgroup of $(\G,F)$ while $Q$ cannot be $e_{\ell}(q)$-closed since $\z^\circ(\G)^F_\ell$ is not contained in $Q$ (see Lemma \ref{lem:Properties of e-closure} (iii)).
\end{rmk}

Given any abelian $\ell$-subgroup $A$ of $\G^F$, we would like to construct a corresponding $e$-closed subgroup of $\G^F$. To do so we first introduce the following notion of \textit{weak} $e$-closure. Recall that given two subgroups $H$ and $K$ of a finite group $G$, the product $HK$ is a subgroup of $G$ if and only if $HK=KH$. The latter condition is definitely satisfied if $[H,K]=1$. Furthermore, in this case, it follows that $HK$ is abelian if and only if so are $H$ and $K$.

\begin{defin}[Weak $e$-closure]
\label{def:Weak e-closure}
Assume that $\ell$ is good for $\G$ and consider a positive integer $e$. We define the \textit{weak $e$-closure} on $\ab_\ell(\G^F)$ as the map
\begin{align*}
\omega_{\ell,e}:\ab_\ell\left(\G^F\right)&\to\ab_\ell\left(\G^F\right)
\\
A &\mapsto A\gamma_{\ell,e}(A)
\end{align*}
which is well-defined according to Lemma \ref{lem:Properties of e-closure} (iii) and the discussion above. We say that an abelian $\ell$-subgroup $A\in\ab_\ell(\G^F)$ is \textit{weakly $e$-closed} if it satisfies $A=\omega_{\ell,e}(A)$ and we denote by $\ab_\ell(\G^F)^{\omega_{\ell,e}}$ the set of weakly $e$-closed $\ell$-subgroups of $\G^F$, that is, the set of $\omega_{\ell,e}$-fixed elements of $\ab_\ell(\G^F)$.
\end{defin}

In the following lemma we collect some basic properties of the weak $e$-closure $\omega_{\ell,e}$ and of the set of weakly $e$-closed $\ell$-subgroups. In particular, for any given abelian $\ell$-subgroup, we show how to produce an $e$-closed abelian $\ell$-subgroup by repeatedly applying the maps $\gamma_{\ell,e}$ and $\omega_{\ell,e}$.

\begin{lem}
\label{lem:Properties of weak e-closure}
Let $A$ be an abelian $\ell$-subgroup of $\G^F$ and assume that $\ell$ is good for $\G$. For every positive integer $e$ we have:
\begin{enumerate}
\item $\omega_{\ell,e}(A)\leq \omega_{\ell,e}(A')$ for every abelian $\ell$-subgroup $A'\in\ab_\ell(\G^F)$ with $A\leq A'$;
\item $\omega_{\ell,e}(A^\alpha)=\omega_{\ell,e}(A)^\alpha$ for every automorphism $\alpha\in\aut_\mathbb{F}(\G^F)$;
\item $A\leq \omega_{\ell,e}(A)$. In particular, there exists a uniquely defined non-negative integer $t_A$ minimal with the property that $\omega_{\ell,e}^{t_A}(A)$ is weakly $e$-closed;
\item if $A$ is weakly $e$-closed, then so is $\gamma_{\ell,e}(A)$. In particular, the $e$-closure restricts to a map
\[\gamma_{\ell,e}:\ab_\ell(\G^F)^{\omega_{\ell,e}}\to\ab_\ell(\G^F)^{\omega_{\ell,e}}\]
and there exists a uniquely defined non-negative integer $r_A$ minimal with the property that $\gamma_{\ell,e}^{r_A}(A)$ is $e$-closed; 
\item if $A$ is $e$-closed, then $A$ is weakly $e$-closed.
\end{enumerate}
\end{lem}

\begin{proof}
If $A\leq A'$ we know from Lemma \ref{lem:Properties of e-closure} (i) that $\gamma_{\ell,e}(A)\leq \gamma_{\ell,e}(A')$ and (i) follows. Similarly, (ii) follows from \ref{lem:Properties of e-closure} (ii). Next, observe that $A\leq A\gamma_{\ell,e}(A)=\omega_{\ell,e}(A)$ for every abelian $\ell$-subgroup and we get an ascending chain of $\ell$-subgroups $\{\omega_{\ell,e}^i(A)\}_{i\geq 0}$ of the finite group $\G^F$. Since this chain must stabilise, there exists a minimal index $t_A\geq 0$ such that $\omega_{\ell,e}^{t_A}(A)=\omega_{\ell,e}^t(A)$ for every $t\geq t_A$. Then, the subgroup $\omega_{\ell,e}^{t_A}(A)$ is weakly $e$-closed by definition.

Assume now that $A$ is weakly $e$-closed. This means that $A=\omega_{\ell,e}(A)=A\gamma_{\ell,e}(A)$ and therefore that $\gamma_{\ell,e}(A)\leq A$. Then, Lemma \ref{lem:Properties of e-closure} (i) implies $\gamma_{\ell,e}(\gamma_{\ell,e}(A))\leq \gamma_{\ell,e}(A)$ and hence $\omega_{\ell,e}(\gamma_{\ell,e}(A))=\gamma_{\ell,e}(A)\gamma_{\ell,e}(\gamma_{\ell,e}(A))=\gamma_{\ell,e}(A)$. This shows that $\gamma_{\ell,e}(A)$ is weakly $e$-closed. As before, we can form a descending chain of $\ell$-subgroups $\{\gamma_{\ell,e}^i(A)\}_{i\geq 0}$ which must eventually stabilise. We can then find a minimal index $r_A\geq 0$ such that $\gamma_{\ell,e}^{r_A}(A)=\gamma_{\ell,e}^t(A)$ for every $t\geq r_A$ and then surely $\gamma_{\ell,e}^{r_A}(A)$ is $e$-closed. This proves (iv) while (v) follows immediately from the definition.
\end{proof}

\subsection{Alternations via $\gamma_{\ell,e}$ and $\omega_{\ell,e}$}
\label{sec:Alternations}

It is well-known that different sets of $\ell$-chains can be considered as equivalent when dealing with the alternating sums appearing in Dade's Conjecture. When the alternating sum is a homotopy invariant, for instance if it is realised as the Euler characteristic of a simplicial complex, these results follow from the existence of certain homotopy equivalences. This approach was first considered by Webb \cite[Section 6]{Web87} and builds on previous results of Bouc, Quillen and Th\'evenaz. A more elementary approach was introduced by Kn\"orr and Robinson in \cite{Kno-Rob89} and is based, implicitly, on the construction of certain alternations in the sense of Definition \ref{def:Alternations} (see the proof of \cite[Proposition 3.3]{Kno-Rob89}). These are then used to show that the alternating sums under consideration can be equivalently stated in terms of different simplicial complexes \cite[Corolary 3.4]{Kno-Rob89}. While this second approach might not offer a conceptual explanation, it is much more flexible and can be used to obtain more precise results (see, for instance, \cite{Rob96}, \cite{Rob00}, \cite{Rob02I}, \cite{Rob02II}, \cite{KLLS} and \cite{Ros-CTC}).

In this section, following these ideas, we construct precise alternations that will be used to obtain cancellation theorems for Dade's Conjecture and the Character Triple Conjecture in the case of finite reductive groups. By using the properties of $\gamma_{\ell,e}$ and $\omega_{\ell,e}$ discussed in the previous section, we are able to remove the contribution of all $\ell$-chains that are not abelian and $e$-closed. This is done in subsequent steps by considering the inclusions
\[\ab_\ell(\G^F)^{\gamma_{\ell,e}}\hookrightarrow\ab_\ell(\G^F)^{\omega_{\ell,e}}\hookrightarrow\ab_\ell(\G^F)\hookrightarrow\mathcal{S}_\ell(\G^F)\]
and assuming suitable hypotheses. In the next section, we then construct a bijection between the remaining $\ell$-chains and the simplicial complex of $e$-split Levi subgroups
\[\Delta\left(\ab_\ell(\G^F)^{\gamma_{\ell,e}}\right)\longleftrightarrow\Delta\left(\CL_e(\G,F)\right)\]
and we reformulate Dade's Conjecture and the Character Triple Conjecture in terms of $e$-split Levi subgroups. Then, we recover these reformulations by the conjectures proposed by the author in \cite{Ros-Generalized_HC_theory_for_Dade}. Thanks to this result we can then apply the machinery of Deligne--Lusztig theory and generalised Harish-Chandra theory to the study of Dade's Conjecture and the Character Triple Conjecture for finite reductive groups in non-defining characteristics.

\begin{lem}
\label{lem:abelian and all subgroups}
Let $G$ be a finite group, $\ell$ a prime dividing the order of $G$ and $Z$ a central $\ell$-subgroup of $G$. For every finite group of automorphisms $A$ of $G$ stabilising $Z$ there is an $A$-alternation of $\Delta(\ab_{\ell,Z}(G,Z))$ in $\Delta(\mathcal{S}_\ell(G,Z))$.
\end{lem}

\begin{proof}
Fix an $\ell$-chain $\sigma=\{Z=P_0<P_1<\dots<P_n\}$ belonging to the set $\Delta(\mathcal{S}_\ell(G,Z))$ but not to $\Delta(\ab_{\ell,Z}(G,Z))$ and observe that the quotient $P_n/Z$ is not abelian, that is, $[P_n,P_n]$ is not contained in $Z$. Moreover, notice that if $Q$ is a maximal subgroup of $P_n$ containing $P_{n-1}$, then $Q$ is normal in $P_n$ and the quotient $P_n/Q$ is abelian. This implies that $[P_n,P_n]\leq Q$ and hence $[P_n,P_n]P_{n-1}\leq Q<P_n$. Next, since $[P_n,P_n]$ is not contained in $Z=P_0$, we can find a unique index $1\leq m\leq n$ such that $[P_n,P_n]$ is contained in $P_m$ but not in $P_{m-1}$. In particular, we have $P_{m-1}<[P_n,P_n]P_{m-1}\leq P_m$. We can then define $\phi(\sigma)$ as the $\ell$-chain obtained from $\sigma$ by adding the term $[P_n,P_n]P_{m-1}$, if $[P_n,P_n]P_{m-1}<P_m$, or by removing the term $P_m$, if $[P_n,P_n]P_{m-1}=P_m$. Observe that the new $\ell$-chain $\phi(\sigma)$ belongs to $\Delta(\mathcal{S}_\ell(G,Z))$ but not to $\Delta(\ab_{\ell,Z}(G,Z))$. In fact, $\phi(\sigma)$ starts with $Z=P_0$ because $m\geq 1$ while the final term of $\phi(\sigma)$ is $P_n$ because $[P_n,P_n]P_{n-1}<P_n$. Therefore, we have shown that the assignment $\sigma\mapsto \phi(\sigma)$ defines a map of $\Delta(\mathcal{S}_\ell(G,Z))\setminus \Delta(\ab_{\ell,Z}(G,Z))$ to itself. Moreover, by the above construction and recalling that $[P_n,P_n]$ is a characteristic subgroup of $P_n$, we further deduce that $\phi(\phi(\sigma))=\sigma$ and that $\phi(\sigma)^\alpha=\phi(\sigma^\alpha)$ for all $\alpha\in A$, while it is clear that $|\phi(\sigma)|=|\sigma|\pm 1$. This shows that $\phi$ is an $A$-alternation of $\Delta(\ab_{\ell,Z}(G,Z))$ in $\Delta(\mathcal{S}_\ell(G,Z))$.
\end{proof}

Now let $\G$, $F$, $q$ and $\ell$ as in the previous sections. We now consider $\ell$-chains whose terms are abelian but not necessarily weakly $e$-closed.

\begin{prop}
\label{prop:weakly e-closed and abelian subgroups}
Assume that $\ell$ is good for $\G$ and consider a positive integer $e$. Then there exists an $\aut_\mathbb{F}(\G^F)$-alternation of $\Delta(\ab_\ell(\G^F,\z^\circ(\G)^F_\ell)^{\omega_{\ell,e}})$ in $\Delta(\ab_\ell(\G^F,\z^\circ(\G)^F_\ell)$.
\end{prop}

\begin{proof}
First observe that every weakly $e$-closed abelian $\ell$-subgroup $A$ contains $Z:=\z^\circ(\G)^F_\ell$ by Lemma \ref{lem:Properties of e-closure} (iii) and that $Z$ is $e$-closed by the definition of $\gamma_{\ell,e}$. Now, let $\sigma=\{Z=A_0<A_1<\dots<A_n\}$ be an $\ell$-chain belonging to $\Delta(\ab_\ell(\G^F,Z))$ but not contained in $\Delta(\ab_\ell(\G^F,Z)^{\omega_{\ell,e}})$. For every $0\leq i\leq n$, define the weakly $e$-closed abelian $\ell$-subgroup $P_i:=\omega_{\ell,e}^{t_{A_i}}(A_i)$ where $t_{A_i}$ is the non-negative integer introduced in Lemma \ref{lem:Properties of weak e-closure} (iii). By the choice of the $\ell$-chain $\sigma$, we know that there exists an index $m$ such that $A_m$ is not weakly $e$-closed, that is, such that $A_m<\omega_{\ell,e}^{t_{A_m}}(A_m)=P_m$. Observe that the above discussion shows in particular that $m\geq 1$. Chose the index $1\leq m$ to be maximal subject to this condition and observe that, whenever $m<n$, the maximality of $m$, together with Lemma \ref{lem:Properties of weak e-closure} (i), implies that $P_m=\omega_{\ell,e}^{t_{A_m}}(A_m)\leq \omega_{\ell,e}^{t_{A_m}}(A_{m+1})=A_{m+1}$ and hence that $A_m<P_m\leq A_{m+1}$. Now, if $m=n$, we define $\phi(\sigma)$ to be the $\ell$-chain obtained by adding the weakly $e$-closed $\ell$-subgroup $P_m$ at the end of the $\ell$-chain $\sigma$. On the other hand, if $m<n$, then we define $\phi(\sigma)$ to be the $\ell$-chain obtain by adding $P_m$ to $\sigma$ if $P_m<A_{m+1}$, and by removing the term $A_{m+1}$ if $P_m=A_{m+1}$. Observe that in any case the new $\ell$-chain $\phi(\sigma)$ has $Z$ as first term and still contains the term $A_m$. Moreover, because $P_m$ is weakly $e$-closed by construction, the index $m$ coincides with the maximal index satisfying the condition in the definition above with respect to the newly defined $\ell$-chain $\phi(\sigma)$. In particular, this shows that $\phi(\sigma)$ is an abelian $\ell$-chain belonging to $\Delta(\ab_\ell(\G^F,Z))$ but not to $\Delta(\ab_{\ell}(\G^F,Z)^{\omega_{\ell,e}})$, and that $\phi(\phi(\sigma))=\sigma$. In addition, Lemma \ref{lem:Properties of weak e-closure} (ii) shows that the map $\phi$ is $\aut_\mathbb{F}(\G^F)$-equivariant while it is clear that $|\phi(\sigma)|=|\sigma|\pm 1$. We can now conclude that $\phi$ is an $\aut_\mathbb{F}(\G^F)$-alternation of $\Delta(\ab_\ell(\G^F,Z)^{\omega_{\ell,e}})$ in $\Delta(\ab_\ell(\G^F,Z))$ as required.
\end{proof}

Before proving the next result we need the following lemma. From now on we assume that $e$ coincides with $e_{\ell}(q)$.

\begin{lem}
\label{lem:e-closed and centre}
Assume that $\ell\in\pi(\G,F)$ and let $e=e_{\ell}(q)$. If $A$ is an abelian $\ell$-subgroup of $\G^F$, then $\gamma_{\ell,e}(A)\leq\z(\G)^F_\ell$ if and only if $A\leq\z(\G)^F_\ell$.
\end{lem}

\begin{proof}
To start, observe that if $A\leq\z(\G)^F_\ell$ then $\c_\G^\circ(A)=\G$ and therefore we immediately obtain $\gamma_{\ell,e}(A)=\z(\G)^F_\ell$. Conversely, assume that $\gamma_{\ell,e}(A)=\z(\G)^F_\ell$ and define $\H:=\c_\G(\z^\circ(\c_\G^\circ(A))_{\Phi_e})$ so that $\z(\G)^F_\ell=\z(\H)^F_\ell$. Now \cite[Proposition 13.19]{Cab-Eng04} implies that the $e$-split Levi subgroup $\H$ of $(\G,F)$ coincides with the centraliser $\c^\circ_\G(\z(\H)^F_\ell)$ so that $\G=\c_\G^\circ(\z(\G)^F_\ell)=\c_\G^\circ(\z(\H)^F_\ell)=\H$ and hence $\z^\circ(\c_\G^\circ(A))_{\Phi_e}\leq \z(\G)$. Then, arguing as in the proof of \cite{Cab-Eng04}, we can apply \cite[Lemma 22.3]{Cab-Eng04} to show that $\z^\circ(\c_\G^\circ(A))_{\Phi_{e\ell^a}}\leq \z(\G)$ for every non-negative integer $a$. This shows that $\G=\c_\G^\circ(A)$ because the centraliser $\c_\G^\circ(A)$ is an $E_{\ell}(q)$-split Levi subgroup of $(\G,F)$ according to Lemma \ref{lem:Abelian subgroups to Levi subgroups} (iii). We finally conclude that $A\leq \z(\G)^F_\ell$ as required above.
\end{proof}

As an immediate consequence we obtain the following corollary for weakly $e$-closed abelian $\ell$-subgroups.

\begin{cor}
\label{cor:e-closed and centre}
Assume that $\ell\in\pi(\G,F)$ and let $e=e_{\ell}(q)$. If $A$ is a weakly $e$-closed abelian $\ell$-subgroup of $\G^F$, then $\gamma_{\ell,e}(A)=\z(\G)^F_\ell$ if and only if $A=\z(\G)^F_\ell$.
\end{cor}

\begin{proof}
This follows immediately from Lemma \ref{lem:e-closed and centre} and Lemma \ref{lem:Properties of e-closure} (iii). In fact, since $\z(\G)^F_\ell\leq \gamma_{\ell,e}(A)$, the condition $\gamma_{\ell,e}(A)\leq \z(\G)^F_\ell$ is equivalent to the equality $\gamma_{\ell,e}(A)=\z(\G)^F_\ell$. On the other hand, when $A$ is weakly $e$-closed, we know that $\z(\G)^F_\ell\leq A$ and so the equality $A=\z(\G)^F_\ell$ is equivalent to the condition $A\leq \z(\G)^F_\ell$.
\end{proof}

Finally, we show how to get rid of those abelian $\ell$-subgroups that are weakly $e$-closed but not $e$-closed.

\begin{prop}
\label{prop:e-closed and weakly e-closed subgroups}
Assume that $\ell\in\pi(\G,F)$ and let $e=e_{\ell}(q)$. Then there exists an $\aut_\mathbb{F}(\G^F)$-alternation of $\Delta(\ab_\ell(\G^F,\z(\G)^F_\ell)^{\gamma_{\ell,e}})$ in $\Delta(\ab_\ell(\G^F,\z(\G)^F_\ell)^{\omega_{\ell,e}})$.
\end{prop}

\begin{proof}
Define $Z:=\z(\G)^F_\ell=\z^\circ(\G)^F_\ell$ and let $\sigma=\{Z=A_0<A_1<\dots<A_n\}$ be an $\ell$-chain belonging to the set $\Delta(\ab_\ell(\G^F,Z)^{\omega_{\ell,e}})$. Since each term $A_i$ of the $\ell$-chain $\sigma$ is weakly $e$-closed, applying Lemma \ref{lem:Properties of weak e-closure} (iv) we can find a non-negative integer $r_{A_i}$ such that the abelian $\ell$-subgroup $P_i:=\gamma_{\ell,e}^{r_{A_i}}(A_i)$ is $e$-closed. Notice also that $P_i$ is contained in $A_i$ since the map $\gamma_{\ell,e}$ restricts to the set of weakly $e$-closed abelian $\ell$-subgroups  and satisfies $\gamma_{\ell,e}(Q)\leq Q\gamma_{\ell,e}(Q)=\omega_{\ell,e}(Q)=Q$ for every $Q\in\ab_\ell(\G^F,Z)^{\omega_{\ell,e}}$. Now, assume that the $\ell$-chain $\sigma$ does not belong to $\Delta(\ab_\ell(\G^F,Z)^{\gamma_{\ell,e}})$. In other words, let us assume that there exists some index $m$ such that $\gamma_{\ell,e}(A_m)<A_m$ or, equivalently, such that $P_m<A_m$. We further choose the index $m$ to be minimal with respect to this property. Firstly, by applying Lemma \ref{lem:e-split produces e-closed} with $\L=\G$, notice that $m>0$ since $Z$ is an $e$-closed $\ell$-subgroup. Secondly, it follows by the minimality of the index $m$ that $A_{m-1}$ is $e$-closed and hence $A_{m-1}=\gamma_{\ell,e}^{r_{A_m}}(A_{m-1})\leq \gamma_{\ell,e}^{r_{A_m}}(A_m)=P_m$ according to Lemma \ref{lem:Properties of weak e-closure} (i). Therefore, we conclude that $A_{m-1}\leq P_m<A_m$. Now, we define a new $\ell$-chain $\phi(\sigma)$ by adding the $\ell$-subgroup $P_m$ to the $\ell$-chain $\sigma$ as an intermediate term between $A_{m-1}$ and $A_m$ if $A_{m-1}<P_m$. On the other hand, if $A_{m-1}=P_m$, then we define the new $\ell$-chain $\phi(\sigma)$ by removing the term $A_{m-1}$ from $\sigma$. Observe that this construction produces an $\ell$-chain $\phi(\sigma)$ that belongs to the set $\Delta(\ab_\ell(\G^F,Z)^{\omega_{\ell,e}})$ but not to $\Delta(\ab_\ell(\G^F,Z)^{\gamma_{\ell,e}})$. In fact, each term of the new $\ell$-chain $\phi(\sigma)$ is weakly $e$-closed by Lemma \ref{lem:Properties of weak e-closure} (iv), the starting term of $\phi(\sigma)$ is $Z$ because the occurrence $Z=P_m$ implies that $A_m=Z$ according to Corollary \ref{cor:e-closed and centre}, and the term $A_m$ is not $e$-closed and belongs to $\phi(\sigma)$. Moreover, observe that $|\phi(\sigma)|=|\sigma|\pm 1$ and that $\phi$ is $\aut_\mathbb{F}(\G^F)$-equivariant. To conclude, we need to show that $\phi^2$ coincides with the identity map on $\Delta(\ab_\ell(\G^F,Z)^{\omega_{\ell,e}})\setminus\Delta(\ab_\ell(\G^F,Z)^{\gamma_{\ell,e}})$. To see this, observe that the construction of $\phi(\sigma)$ only modifies the chain $\sigma$ by adding or removing $e$-closed abelian $\ell$-subgroups strictly contained in $A_m$. Therefore, the minimal term of the $\ell$-chain $\phi(\sigma)$ that is not $e$-closed is once again $A_m$ and hence $\phi(\phi(\sigma))=\sigma$. This shows that $\phi$ defines an $\aut_\mathbb{F}(\G^F)$-alternation of $\Delta(\ab_\ell(\G^F,Z)^{\gamma_{\ell,e}})$ in $\Delta(\ab_\ell(\G^F,Z)^{\omega_{\ell,e}})$. 
\end{proof}

\section{Dade's Conjecture and the Character Triple Conjecture}
\label{sec:Dade and CTC}

We now apply the results obtained in Section \ref{sec:Adapting ell-groups to e-split Levi}, and in particular the alternations described in Section \ref{sec:Alternations}, to prove cancellation theorems for Dade's Conjecture and for the Character Triple Conjecture in the case of finite reductive groups. These cancellation theorems allow us to show that, under suitable hypotheses, the conjectures introduced by the author in \cite{Ros-Generalized_HC_theory_for_Dade} imply Dade's Conjecture and its inductive condition in the form of the Character Triple Conjecture. The results of this section improve those of \cite[Section 7.3]{Ros-Generalized_HC_theory_for_Dade} where the simpler case of large primes was considered.

\subsection{Statement of the conjectures} 

Let $G$ be a finite group and consider an $\ell$-block $B$ of $G$, a non-negative integer $d$, and an irreducible character $\lambda$ of a fixed subgroup $Z$ of $\z(G)$. For every chain $\sigma\in\Delta(\mathcal{S}_\ell(G,Z_\ell))$ and every $\ell$-block $b$ of the stabiliser $G_\sigma$, observe that the Brauer induced block $b^G$ is defined according to \cite[Lemma 3.2]{Kno-Rob89}. Moreover, notice that $Z$ stabilises the chain $\sigma$ and hence is contained in $G_\sigma$. We define $\k^d(B_\sigma,\lambda)$ to be the number of irreducible characters $\vartheta$ of $G_\sigma$ lying above $\lambda$, with $\ell$-defect $d$, and such that $\bl(\vartheta)^G=B$. Here $\bl(\vartheta)$ denotes the unique $\ell$-block of $G_\sigma$ that contains $\vartheta$.

\begin{dade}[{{\cite[Conjecture 15.5]{Dad94}}}]
\label{conj:Dade}
Let $G$ be a finite group and consider a subgroup $Z$ of $\z(G)$. Then
\[\sum\limits_{\sigma\in\Delta(\mathcal{S}_\ell(G,Z_\ell))/G}(-1)^{|\sigma|}\k^d(B_\sigma,\lambda)=0\]
for every $\ell$-block $B$ of $G$ whose defect groups strictly contain $Z_\ell$, every non-negative integer $d$, and every irreducible character $\lambda$ of $Z$.
\end{dade}

The ordinary version of Dade's Conjecture can be recovered by setting $Z=1$. Moreover, by using a standard argument on the contractibility of the simplicial complex $\Delta(\mathcal{S}_\ell(G,Z_\ell))$ due to Quillen, the above alternating sum vanishes for trivial reasons unless $Z_\ell$ coincides with $\O_\ell(G)$.

In \cite{Spa17}, Dade's Projective Conjecture \ref{conj:Dade} was reduced to a statement about quasi-simple groups known as the \textit{inductive condition for Dade's Conjecture} (see \cite[Definition 6.7]{Spa17}). This condition is stated in terms of the Character Triple Conjecture that we now describe. Let $G$ be a finite group, $B$ an $\ell$-block of $G$, and $d$ a non-negative integer as above. Assume now that $Z$ is an $\ell$-subgroup central in $G$. Define the set $\C^d(B,Z)$ of pairs $(\sigma,\vartheta)$ where $\sigma$ is a chain belonging to the simplicial complex $\Delta(\mathcal{S}_\ell(G,Z))$ and $\vartheta$ is an irreducible character of the stabiliser $G_\sigma$ with $\ell$-defect $d$ and such that $\bl(\vartheta)^G=B$. We write $\C^d(B)$ to denote $\C^d(B,1)$. Observe that, while in previous papers (see, for instance, \cite{Spa17} and \cite{Ros22}) this notation was reserved for the set $\C^d(B,\O_\ell(G))$, no confusion can arise as we can always assume that $\O_\ell(G)$ is central (see \cite[Lemma 2.3]{Ros22}) and the final results of this paper assume that the order of $\z(G)$ is prime to $\ell$. Next, partition $\C^d(B,Z)$ into its subsets $\C^d(B,Z)_\pm$ consisting of those pairs $(\sigma,\vartheta)$ such that $(-1)^{|\sigma|}=\pm 1$. Furthermore, notice that $G$ acts by conjugation on $\C^d(B,Z)_\pm$ and denote by $\C^d(B,Z)_\pm/G$ the corresponding set of $G$-orbits $\overline{(\sigma,\vartheta)}$. In what follows we freely use the notion of $G$-block isomorphism of character triples, denoted by $\iso{G}$, as introduced in \cite[Definition 3.6]{Spa17}. We refer the reader to that paper for further details.

\begin{ctc}[{{\cite[Conjecture 6.3]{Spa17}}}]
\label{conj:CTC}
Let $G$ be a finite group and consider a central $\ell$-subgroup $Z$ of $G$ and an $\ell$-block $B$ of $G$ with defect groups strictly containing $Z$. Suppose that $G\unlhd X$. Then, for every non-negative integer $d$, there exists an $\n_X(Z)_B$-equivariant bijection
\[\Omega:\C^d(B,Z)_+/G\to\C^d(B,Z)_-/G\]
such that
\[\left(X_{\sigma,\vartheta},G_\sigma,\vartheta\right)\iso{G}\left(X_{\rho,\chi},G_\rho,\chi\right)\]
for every $(\sigma,\vartheta)\in\C^d(B,Z)_+$ and $(\rho,\chi)\in\Omega(\overline{(\sigma,\vartheta)})$.
\end{ctc}

Next, we describe \cite[Conjecture C and Conjecture D]{Ros-Generalized_HC_theory_for_Dade}. As in the previous sections, let $(\G,F)$ be a finite reductive group defined over the finite field $\mathbb{F}_q$ and let $\ell$ be a prime not dividing $q$. From now on, and for the rest of this section, assume that $e=e_{\ell}(q)$ as defined in Section \ref{sec:Order polynomial}. Next, consider an $\ell$-block $B$ of $\G^F$, a non-negative integer $d$, and an irreducible character $\lambda$ of a subgroup $Z$ of $\z(\G)^F$. Let $\k^d(B,\lambda)$ be the number of irreducible characters contained in the $\ell$-block $B$, lying above $\lambda$, and with $\ell$-defect $d$. Similarly, let $\k^d_{\rm c}(B,\lambda)$ denote the number of such characters that are additionally $e$-cuspidal (see \cite[Definition 3.5.19]{Gec-Mal20}). For every $\sigma\in\Delta(\CL_e^\star(\G,F))$, define $\k^d(B_\sigma,\lambda)$ as the number of irreducible characters $\vartheta$ of the stabiliser $\G_\sigma^F$ lying above $\lambda$, with $\ell$-defect $d$, and such that $\bl(\vartheta)^G$ is defined and coincides with $B$. Although the latter condition on block induction differs from the one given in \cite[Section 5.1]{Ros-Generalized_HC_theory_for_Dade}, the two are equivalent under the hypotheses considered below thanks to \cite[Lemma 5.5]{Ros-Generalized_HC_theory_for_Dade}. Furthermore, here we are considering a slightly more general statement by allowing the presence of the central character $\lambda$ as considered in the projective version of Dade's Conjecture.

\begin{conj}[{{\cite[Conjecture C]{Ros-Generalized_HC_theory_for_Dade}}}]
\label{conj:Dade reductive}
Consider an $\ell$-block $B$ of $\G^F$, a non-negative integer $d$, and an irreducible character $\lambda$ of a subgroup $Z$ of $\z(\G)^F$. Then
\[k^d(B,\lambda)=k^d_{\rm c}(B,\lambda)+\sum\limits_{\rho\in\Delta(\CL_e^\star(\G,F))/\G^F}(-1)^{|\rho|+1}k^d(B_\rho,\lambda).\]
\end{conj}

Now, define the set $\CP_e(B)$ of $e$-cuspidal pairs $(\M,\mu)$ (see \cite[Definition 3.5.19]{Gec-Mal20}) such that $\bl(\mu)^{\G^F}$ is defined and coincides with $B$, and consider its subset $\CP_e^\star(B)$ consisting of those such pairs such that $\M$ is strictly contained in $\G$. Moreover, let $\ab(\mu)$ be the set of irreducible characters of $\M^F$ of the form $\nu\mu$ for some linear character $\nu$ of $\M^F$. Finally, define the set $\CL^d_e(B)$ of quadruples $(\sigma,\M,\ab(\mu),\vartheta)$ where $\sigma$ is a chain of the simplicial complex $\Delta(\CL_e(\G,F))$ with smallest term $\L$ (in this case $\L$ is contained in each term of $\sigma$ and hence normalizes the chain), $(\M,\mu)$ is an element of the set $\CP^\star_e(B)$ with $\M\leq \L$, and $\vartheta$ is an irreducible character of the stabiliser $\G_\sigma^F$ with $\ell$-defect $d$, whose block $\bl(\vartheta)$ induces to $\G^F$ and satisfies $\bl(\vartheta)^{\G^F}=B$, and lying above some character in an $e$-Harish-Chandra series $\E(\L^F,(\M,\mu'))$ (see \cite[3.5.24]{Gec-Mal20}) with $\mu'\in\ab(\mu)$. We refer the reader to \cite[Section 5.2]{Ros-Generalized_HC_theory_for_Dade}) for further details. As for the Character Triple Conjecture, we partition $\CL_e^d(B)$ into two subsets $\CL_e^d(B)_\pm$ according to the parity of chains and denote by $\CL_e^d(B)_\pm/\G^F$ the corresponding sets of $\G^F$-orbits.

\begin{conj}[{{\cite[Conjecture D]{Ros-Generalized_HC_theory_for_Dade}}}]
\label{conj:CTC reductive}
For every $\ell$-block $B$ of $\G^F$ and every non-negative integer $d$, there exists an $\aut_\mathbb{F}(\G^F)_B$-equivariant bijection
\[\Lambda:\CL_e^d(B)_+/\G^F\to\CL_e^d(B)_-/\G^F\]
such that
\[\left(X_{\sigma,\vartheta},\G^F_\sigma,\vartheta\right)\iso{\G^F}\left(X_{\rho,\chi},\G_\rho^F,\chi\right)\]
for every $(\sigma,\M,\ab(\mu),\vartheta)\in\CL_e^d(B)_+$ and $(\rho,\K,\ab(\kappa),\chi)\in\CL_e^d(B)_-$ whose $\G^F$-orbits correspond under $\Lambda$, and where $X:=\G^F\rtimes \aut_\mathbb{F}(\G^F)$.
\end{conj}

Before proceeding to the main results of this section, notice that Conjecture \ref{conj:Dade reductive} is a consequence of Conjecture \ref{conj:CTC reductive} (see \cite[Theorem E]{Ros-Generalized_HC_theory_for_Dade}) and that the latter has been reduced to certain extendibility conditions for characters of $e$-split Levi subgroups \cite{Ros-Clifford_automorphisms_HC}. More recently, adaptations of these conjectures to the subset of unipotent characters have been verified in \cite{Ros-Unip} for groups of types $\bf{A}$, $\bf{B}$ and $\bf{C}$.

\subsection{Cancellation theorems for finite reductive groups}
\label{sec:Cancellation}

Combining the alternations constructed in Section \ref{sec:Alternations} we obtain cancellation theorems for Dade's Projective Conjecture \ref{conj:Dade} and the Character Triple Conjecture \ref{conj:CTC}. More precisely, we show that the only chains that contribute to the alternating sums, and bijections, predicted by these conjectures are those whose terms are $e$-closed abelian $\ell$-subgroups. We keep $\G$, $F$, $q$ and $\ell$ as in the previous sections and recall that we are assuming $e=e_{\ell}(q)$. First, we consider the alternating sum from Dade's Projective Conjecture \ref{conj:Dade}.

\begin{theo}
\label{thm:Cancellation theorem for Dade}
Let $(\G,F)$ be a finite reductive group and consider a prime $\ell\in\pi(\G,F)$ not dividing the order of $\z(\G)^F$. Fix an $\ell$-block $B$ of $\G^F$, a non-negative integer $d$, and an irreducible character $\lambda$ of a subgroup $Z$ of $\z(\G)^F$. Then
\begin{equation}
\label{eq:Cancellation theorem for Dade}
\sum\limits_{\sigma\in\Delta(\mathcal{S}_\ell(\G^F))/\G^F}(-1)^{|\sigma|}\k^d(B_\sigma,\lambda)=\sum\limits_{\rho\in\Delta(\ab_\ell(\G^F)^{\gamma_{\ell,e}})/\G^F}(-1)^{|\rho|}\k^d(B_\rho,\lambda).
\end{equation}
\end{theo}

\begin{proof}
Under the above hypotheses the subgroups $Z$ and $\z(\G)^F_\ell$ are trivial and hence we have $\ab_\ell(\G^F,\z(\G)^F_\ell)=\ab_\ell(\G^F)=\ab_{\ell,Z}(\G^F,Z)$ and $\mathcal{S}_\ell(\G^F,Z)=\mathcal{S}_\ell(\G^F)$. Then, by applying Lemma \ref{lem:abelian and all subgroups}, Proposition \ref{prop:weakly e-closed and abelian subgroups}, and Proposition \ref{prop:e-closed and weakly e-closed subgroups} respectively, we obtain a $\G^F$-alternation $\phi_1$ of the simplicial complex of abelian $\ell$-subgroups $\Delta(\ab_\ell(\G^F))$ in the simplicial complex of all $\ell$-subgroups $\Delta(\mathcal{S}_\ell(\G^F))$, a $\G^F$-alternation $\phi_2$ of the simplicial complex of weakly $e$-closed abelian $\ell$-subgroups $\Delta(\ab_\ell(\G^F)^{\omega_{\ell,e}})$ in $\Delta(\ab_\ell(\G^F))$, and a $\G^F$-alternation $\phi_3$ of the simplicial complex of $e$-closed abelian $\ell$-subgroups $\Delta(\ab_\ell(\G^F)^{\gamma_{\ell,e}})$ in $\Delta(\ab_\ell(\G^F)^{\omega_{\ell,e}})$. Now, combining $\phi_1$, $\phi_2$ and $\phi_3$ we obtain a $\G^F$-alternation of the simplicial complex of $e$-closed abelian $\ell$-subgroups $\Delta(\ab_\ell(\G^F)^{\gamma_{\ell,e}})$ in the simplicial complex of all $\ell$-subgroups $\Delta(\mathcal{S}_\ell(\G^F))$. More precisely, we define the map
\[\phi:\Delta\left(\mathcal{S}_\ell\left(\G^F\right)\right)\setminus \Delta\left(\ab_\ell\left(\G^F\right)^{\gamma_{\ell,e}}\right)\to \Delta\left(\mathcal{S}_\ell\left(\G^F\right)\right)\setminus \Delta\left(\ab_\ell\left(\G^F\right)^{\gamma_{\ell,e}}\right)\]
by setting
\[\phi(\sigma):=\begin{cases}
\phi_1(\sigma), & \text{if }\sigma\in \Delta\left(\mathcal{S}_\ell\left(\G^F\right)\right)\setminus \Delta\left(\ab_\ell\left(\G^F\right)\right)
\\
\phi_2(\sigma), & \text{if }\sigma\in \Delta\left(\ab_\ell\left(\G^F\right)\right)\setminus \Delta\left(\ab_\ell\left(\G^F\right)^{\omega_{\ell,e}}\right)
\\
\phi_3(\sigma), & \text{if }\sigma\in \Delta\left(\ab_\ell\left(\G^F\right)^{\omega_{\ell,e}}\right)\setminus \Delta\left(\ab_\ell\left(\G^F\right)^{\gamma_{\ell,e}}\right)
\end{cases}\]
for every simplex $\sigma$ belonging to $\Delta(\mathcal{S}_\ell(\G^F))$ but outside $\Delta(\ab_\ell(\G^F)^{\gamma_{\ell,e}})$. From this definition it is immediate to verify that the map $\phi$ satisfies the requirements of Definition \ref{def:Alternations}. Next, we define the $\G^F$-stable function
\[f_{B,d,\lambda}(H):=\begin{cases}
\left|\left\lbrace\enspace\psi\in\irr(H)\middle|\enspace \substack{ \bl(\psi)^G \text{ is defined and equal to }B, \\ \psi\text{ has }\ell\text{-defect }d\text{ and lies above }\lambda}\right\rbrace\right|, & \text{if }\z(\G)^F\leq H
\\
0, & \text{if }\z(\G)^F\nleq H
\end{cases}\]
for every subgroup $H$ of $\G^F$. Since the centre $\z(\G)^F$ is contained in every stabiliser $\G^F_\sigma$ and block induction (in the sense of Brauer) is defined from $\G_\sigma^F$ to $\G^F$ (see \cite[Lemma 3.2]{Kno-Rob89}), observe that $f_{B,d,\lambda}(\G^F_\sigma)$ coincides with $\k^d(B_\sigma,\lambda)$ for every $\sigma\in\Delta(\mathcal{S}_\ell(\G^F))$. The result now follows by applying Lemma \ref{lem:Alternation and G-stable functions} with $f=f_{B,d,\lambda}$ and considering the simplicial complexes $\mathcal{X}:=\Delta(\mathcal{S}_\ell(\G^F))$ and $\mathcal{X}':=\Delta(\ab_\ell(\G^F)^{\gamma_{\ell,e}})$.
\end{proof}

Next, we consider the bijection from the Character Triple Conjecture. For this purpose we need to introduce some further notation. More precisely, let $\C^d_{\gamma_{\ell,e}}(B)$ denote the subset of $\C^d(B)$ consisting of those pairs $(\sigma,\vartheta)$ with $\sigma$ belonging to $\Delta(\ab_\ell(\G^F)^{\gamma_{\ell,e}})$. We then define $\C^d_{\gamma_{\ell,e}}(B)^{\rm c}$ to be the complement of $\C^d_{\gamma_{\ell,e}}(B)$ inside $\C^d(B)$. Notice that $\aut_\mathbb{F}(\G^F)_B$ acts on these two sets because $\ab_\ell(\G^F)^{\gamma_{\ell,e}}$ is an $\aut_\mathbb{F}(\G^F)$-stable subset of $\mathcal{S}_\ell(\G^F)$. Then, as usual, we define $\C^d_{\gamma_{\ell,e}}(B)_\pm:=\C^d_{\gamma_{\ell,e}}(B)\cap \C^d(B)_\pm$ and denote by $\C^d_{\gamma_{\ell,e}}(B)_\pm/\G^F$ the corresponding set of $\G^F$-orbits. A similar notation is used for the complement set $\C^d_{\gamma_{\ell,e}}(B)^{\rm c}$. With this notation we have the following theorem.

\begin{theo}
\label{thm:Cancellation theorem for CTC}
Let $(\G,F)$ be a finite reductive group and consider a prime $\ell\in\pi(\G,F)$ not dividing the order of $\z(\G)^F$. Fix an $\ell$-block $B$ of $\G^F$ and a non-negative integer $d$. Then there exists an $\aut_\mathbb{F}(\G^F)_B$-equivariant bijection
\[\Upsilon:\C^d_{\gamma_{\ell,e}}(B)^{\rm c}_+/\G^F\to \C^d_{\gamma_{\ell,e}}(B)^{\rm c}_-/\G^F\]
such that
\[\left(X_{\sigma,\vartheta},\G^F_\sigma,\vartheta\right)\iso{\G^F}\left(X_{\rho,\chi},\G^F_\rho,\chi\right)\]
for every $(\sigma,\vartheta)\in\C^d_{\gamma_{\ell,e}}(B)_+^{\rm c}$ and $(\rho,\chi)\in\Upsilon(\overline{(\sigma,\vartheta)})$ and where $X:=\G^F\rtimes \aut_\mathbb{F}(\G^F)$.
\end{theo}

\begin{proof}
As in the proof of Theorem \ref{thm:Cancellation theorem for Dade} we can construct an $\aut_\mathbb{F}(\G^F)$-alternation $\phi$ of the simplicial complex of $e$-closed abelian $\ell$-subgroups $\Delta(\ab_\ell(\G^F)^{\gamma_{\ell,e}})$ in the simplicial complex of all $\ell$-subgroups $\Delta(\mathcal{S}_\ell(\G^F))$. We define the map
\[\Upsilon':\C^d_{\gamma_{\ell,e}}(B)_+^{\rm c}\to\C^d_{\gamma_{\ell,e}}(B)_-^{\rm c}\]
by setting
\[\Upsilon'(\sigma,\vartheta):=(\phi(\sigma),\vartheta)\]
for every $(\sigma,\vartheta)\in\C_{\gamma_{\ell,e}}^d(B)_+$. Observe that this map is a well-defined $\aut_{\mathbb{F}}(\G^F)_B$-equivariant bijection because of the properties of $\phi$. In fact, since $\phi$ is $\aut_\mathbb{F}(\G^F)$-equivariant (and in particular $\G^F$-equivariant), we know that $\G^F_\sigma=\G^F_{\phi(\sigma)}$ so that $\vartheta$ is an irreducible character of $\G^F_{\phi(\sigma)}$. Moreover, since $|\phi(\sigma)|=|\sigma|\pm 1$, we deduce that the pair $(\phi(\sigma),\vartheta)$ belongs to $\C_{\gamma_{\ell,e}}^d(B)_-^{\rm c}$. Notice also that $\Upsilon'$ is a bijection: the inverse of $\Upsilon'$ is easily described as $(\Upsilon')^{-1}(\phi(\sigma),\vartheta):=(\phi(\phi(\sigma)),\vartheta)=(\sigma,\vartheta)$. In addition, again by using the equivariance properties of $\phi$, observe that $X_\sigma=X_{\phi(\sigma)}$ and therefore that
\[\left(X_{\sigma,\vartheta},\G^F_\sigma,\vartheta\right)\iso{\G^F}\left(X_{\phi(\sigma),\vartheta},\G^F_{\phi(\sigma)},\vartheta\right).\]
If we now define $\Upsilon$ to be the map induced by $\Upsilon'$ on the corresponding sets of $\G^F$-orbits, we obtain a bijection with the properties required in the statement.
\end{proof}

\subsection{Replacing $\ell$-chains with $e$-chains}

Our aim is now to use the cancellation theorems described in the previous section in order to establish a connection between Dade's Projective Conjecture \ref{conj:Dade} and Conjecture \ref{conj:Dade reductive} as well as between the Character Triple Conjecture \ref{conj:CTC} and Conjecture \ref{conj:CTC reductive}. First, by applying the characterisation of $e$-closed abelian $\ell$-subgroups given in Lemma \ref{lem:e-split produces e-closed}, we construct an equivariant bijection between the set of $e$-chains belonging to $\Delta(\CL_e(\G,F))$ and the set of $\ell$-chains belonging to $\Delta(\ab_\ell(\G^F,\z(\G)^F_\ell)^{\gamma_{\ell,e}})$, that is, $\ell$-chains consisting of $e$-closed abelian $\ell$-subgroups. Keep $\G$, $F$, $q$ and $\ell$ as above and assume $e=e_{\ell}(q)$.

\begin{lem}
\label{lem:e-split and e-closed subgroups}
Assume that $\ell$ is good for $\G$ and does not divide $|\z(\G)^F:\z^\circ(\G)^F|$. Then the maps
\begin{align*}
\iota:\Delta\left(\CL_e(\G,F)\right)&\to \Delta\left(\ab_\ell\left(\G^F,\z(\G)^F_\ell\right)^{\gamma_{\ell,e}}\right)
\\
\left\lbrace\L_i\right\rbrace_i &\mapsto \left\lbrace\z(\L_i)^F_\ell\right\rbrace_i
\end{align*}
and 
\begin{align*}
\delta:\Delta\left(\ab_\ell\left(\G^F,\z(\G)^F_\ell\right)^{\gamma_{\ell,e}}\right) &\to \Delta\left(\CL_e(\G,F)\right)
\\
\left\lbrace A_i\right\rbrace_i &\mapsto \left\lbrace\c_\G\left(\z^\circ\left(\c^\circ_\G(A_i)\right)_{\Phi_e}\right)\right\rbrace_i
\end{align*}
are $\aut_\mathbb{F}(\G^F)$-equivariant bijections satisfying $\delta=\iota^{-1}$, and with $|\iota(\sigma)|=|\sigma|$ and $|\delta(\rho)|=|\rho|$ for every $e$-chain $\sigma$ belonging to $\Delta(\CL_e(\G,F))$ and every $\ell$-chain $\rho$ belonging to $\Delta(\ab_\ell(\G^F,\z(\G)^F_\ell)^{\gamma_{\ell,e}})$.
\end{lem}

\begin{proof}
Notice that $\z(\L)^F_\ell=\z^\circ(\L)^F_\ell$ is an $e$-closed abelian $\ell$-subgroup for every $e$-split Levi subgroup $\L$ of $(\G,F)$ as explained in Lemma \ref{lem:e-split produces e-closed}. Thus, given an $e$-chain $\sigma\in\Delta(\CL_e(\G,F))$, and recalling that $\G$ is the initial term of $\sigma$, we deduce that $\iota(\sigma)$ belongs to $\Delta(\ab_\ell(\G^F,\z(\G)^F_\ell)^{\gamma_{\ell,e}})$. On the other hand, the centraliser $\c_\G(\z^\circ(\c_\G^\circ(A))_{\Phi_e})$ is an $e$-split Levi subgroup of $(\G,F)$ for every $\ell$-subgroup $A$ of $\G^F$. Since $\c_\G(\z^\circ(\c_\G^\circ(\z(\G)^F_\ell))_{\Phi_e})=\G$, we deduce that $\delta(\rho)$ belongs to $\Delta(\CL_e(\G,F))$ whenever $\rho$ is an $\ell$-chain of $\G^F$ with starting term $\z(\G)^F_\ell$, and in particular whenever $\rho\in\Delta(\ab_\ell(\G^F,\z(\G)^F_\ell)^{\gamma_{\ell,e}})$. This shows that $\iota$ and $\delta$ are well-defined. It is also immediate to check that both maps commute with the action of $\aut_\mathbb{F}(\G^F)$.

Next, consider an $e$-chain $\sigma\in\Delta(\CL_e(\G,F))$ and observe that the terms of $\delta(\iota(\sigma))$ are given by $\c_\G(\z^\circ(\c_\G^\circ(\z(\L_i)^F_\ell))_{\Phi_e})$. But the latter coincides with $\L_i$ according to Lemma \ref{lem:Abelian subgroup of e-split Levi subgroup}. Conversely, if $\rho$ is an $\ell$-chain belonging to $\Delta(\ab_\ell(\G^F,\z(\G)^F_\ell)^{\gamma_{\ell,e}})$, then the terms of $\iota(\delta(\rho))$ are given by $\z^\circ(\c_\G(\z^\circ(\c_\G^\circ(A_i))_{\Phi_e}))^F_\ell=\gamma_{\ell,e}(A_i)=A_i$ where the last equality follows from the fact that $A_i$ is $e$-closed. This shows that $\delta=\iota^{-1}$.

To conclude, we need to show that both $\iota$ and $\delta$ preserve the length of chains. If $\L$ and $\L'$ are $e$-split Levi subgroups of $(\G,F)$ with $\L<\L'$, then we must have $\z(\L')^F_\ell<\z(\L)^F_\ell$ according to \cite[Proposition 13.19]{Cab-Eng04}. Furthermore, if $A$ and $A'$ are $e$-closed abelian $\ell$-subgroups with $A<A'$, then we have $\c_\G(\z^\circ(\c_\G^\circ(A))_{\Phi_e})>\c_\G(\z^\circ(\c_\G^\circ(A'))_{\Phi_e})$. In fact, notice that $\c_\G(\z^\circ(\c_\G^\circ(A))_{\Phi_e})\geq \c_\G(\z^\circ(\c_\G^\circ(A'))_{\Phi_e})$ and that we cannot have an equality since this would imply $A=\gamma_{\ell,e}(A)=\gamma_{\ell,e}(A')=A'$. It follows that $\iota$ and $\delta$ preserve the length of chains as required.
\end{proof}

In the subsequent proofs we make use of the following generalisation of \cite[Proposition 4.18]{Ros-Generalized_HC_theory_for_Dade}.

\begin{lem}
\label{lem:Defect zero characters and cuspidality}
Let $(\G,F)$ be a finite reductive group with dual $(\G^*,F^*)$ and consider a prime $\ell\in\pi(\G^*,F^*)$ not dividing the order of $\z(\G^*)^{F^*}$. If $\chi$ is an $e$-cuspidal character of $\G^F$, then $\chi$ has $\ell$-defect zero.
\end{lem}

\begin{proof}
Let $s$ be a semisimple element of $\G^{*F^*}$ whose (rational) Lusztig series contains the character $\chi$ (see \cite[Definition 2.6.1 and Theorem 2.6.2]{Gec-Mal20}). Proceeding as in the proof of \cite[Proposition 4.18]{Ros-Generalized_HC_theory_for_Dade}, and noticing that the first part of that argument can be carried out under our assumption, we deduce that $\chi(1)_\ell=|\G^F|_\ell/|\z(\H)^F|_\ell$ where $\H:=\c_{\G^*}^\circ(s)$. It remains to prove that $\ell$ does not divide the order of $\z(\H)^F$, or equivalently, that $Z:=\z(\H)^F_\ell=1$. First, observe that $\z^\circ(\G^*)_{\Phi_e}=\z^\circ(\H)_{\Phi_e}$ according to \cite[Proposition 1.10]{Cab-Eng99} and therefore that $\G^*$ is the unique $e$-split Levi subgroup of $(\G^*,F^*)$ containing $\H$. On the other hand, basic properties of Levi subgroups show that $\H$ is contained in $\K:=\c_{\G^*}(\z^\circ(\c_{\G^*}^\circ(Z))_{\Phi_e})$ and that $\K$ is an $e$-split Levi subgroup of $(\G^*,F^*)$. It follows that $\K=\G^*$ and therefore that $\z(\K)^F_\ell=\z(\G^*)^F_\ell$. Finally, by noticing that $\z(\K)^F_\ell=\gamma_{\ell,e}(Z)$, Lemma \ref{lem:e-closed and centre} implies that $Z$ is contained in $\z(\G^*)_\ell^{F^*}$ and, since the latter is trivial by assumption, we deduce that $Z=1$ as wanted.
\end{proof}

We now apply Theorem \ref{thm:Cancellation theorem for Dade} and Lemma \ref{lem:e-split and e-closed subgroups} to reformulate the alternating sum from Dade's Projective Conjecture \ref{conj:Dade} in terms of chains in the simplicial complex $\Delta(\CL_e(\G,F))$.

\begin{prop}
\label{prop:Replacement theorem for Dade}
Let $(\G,F)$ be finite reductive group and consider a prime $\ell\in\pi(\G,F)$ not dividing the order of $\z(\G)^F$. Fix an $\ell$-block $B$ of $\G^F$, a non-negative integer $d$, and an irreducible character $\lambda$ of a subgroup $Z$ of $\z(\G)^F$. Then 
\begin{equation}
\label{eq:Replacement theorem for Dade}
\sum\limits_{\sigma\in\Delta(\mathcal{S}_\ell(\G^F))/\G^F}(-1)^{|\sigma|}\k^d(B_\sigma,\lambda)=\sum\limits_{\rho\in\Delta(\CL_e(\G,F))/\G^F}(-1)^{|\rho|}\k^d(B_\rho,\lambda)
\end{equation}
\end{prop}

\begin{proof}
By applying Theorem \ref{thm:Cancellation theorem for Dade} we can replace the alternating sum on the left-hand side of the equality \eqref{eq:Replacement theorem for Dade} with the corresponding sum over $\G^F$-conjugacy classes of chains of $e$-closed abelian $\ell$-subgroups, that is, over the $\G^F$-orbits in the simplicial complex $\Delta(\ab_\ell(\G^F)^{\gamma_{\ell,e}})$ (see \eqref{eq:Cancellation theorem for Dade}). Next, notice that under our assumption we have $\z(\G)^F_\ell=1$ and therefore Lemma \ref{lem:e-split and e-closed subgroups} exhibits a bijection $\delta$ between $\Delta(\ab_\ell(\G^F)^{\gamma_{\ell,e}})$ and the simplicial complex $\Delta(\CL_e(\G,F))$ corresponding to the poset of $e$-split Levi subgroups of $(\G,F)$. Furthermore, the equivariance properties described in Lemma \ref{lem:e-split and e-closed subgroups} imply that the bijection $\delta$ preserves $\G^F$-orbits and stabilisers, i.e. we have $\G^F_\sigma=\G^F_\rho$ for every $\sigma\in\Delta(\ab_\ell(\G^F)^{\gamma_{\ell,e}})$ and $\rho=\delta(\sigma)$. In particular, this shows that $\k^d(B_\sigma,\lambda)$ coincides with $\k^d(B_\rho,\lambda)$ whenever $\delta(\sigma)=\rho$. Since we also know that $|\delta(\sigma)|=|\sigma|$, it follows that the alternating sum over $\G^F$-orbits of chains of abelian $e$-closed $\ell$-subgroups coincides with that over $\G^F$-orbits of chains of $e$-split Levi subgroups of $(\G,F)$ and hence we obtain the equality $\eqref{eq:Replacement theorem for Dade}$.
\end{proof}

Finally, we can show the equivalence of Dade's Projective Conjecture \ref{conj:Dade} and Conjecture \ref{conj:Dade reductive} under suitable hypothesis. Our next result extends \cite[Proposition 7.10]{Ros-Generalized_HC_theory_for_Dade}.

\begin{theo}
\label{thm:Replacement theorem for Dade}
Let $(\G,F)$ be finite reductive group and consider a prime $\ell\in\pi(\G,F)\cap \pi(\G^*,F^*)$ not dividing the order of $\z(\G)^F$ nor that of $\z(\G^*)^{F^*}$. Then the following statements are equivalent for every $\ell$-block $B$ of $\G^F$ with non-trivial defect, every non-negative integer $d$, and every character $\lambda$ of a subgroup $Z$ of $\z(\G)^F$:
\begin{enumerate}
\item Dade's Projective Conjecture \ref{conj:Dade} holds for $B$, $d$, and $\lambda$;
\item Conjecture \ref{conj:Dade reductive} holds for $B$, $d$, and $\lambda$. 
\end{enumerate}
\end{theo}

\begin{proof}
To start, we notice that under our assumptions $Z_\ell=1$ and hence Dade's Projective Conjecture \ref{conj:Dade} holds for $B$, $d$, and $\lambda$ if and only if the alternating sum on the left-hand side of \eqref{eq:Replacement theorem for Dade} is equal to zero. Therefore, Proposition \ref{prop:Replacement theorem for Dade} shows that (i) is equivalent to the vanishing of the right-hand side of \eqref{eq:Replacement theorem for Dade}. We now separate the contributions to that alternating sum according to the partition $\Delta(\CL_e(\G,F))=\Delta(\CL_e^\star(\G,F))\cup\{\rho_0\}$ and where $\rho_0=\{\G\}$. Since $\G^F_{\rho_0}=\G^F$ it follows that $\k^d(B_{\rho_0},\lambda)=\k^d(B,\lambda)$ and therefore (i) is equivalent to the equality
\[\sum\limits_{\rho\in\Delta(\CL_e^\star(\G,F))/\G^F}(-1)^{|\rho|}\k^d(B_\rho,\lambda)+\k^d(B,\lambda)=0\]
that we rewrite as
\begin{equation}
\label{eq:Replacement theorem for Dade 1}
\k^d(B,\lambda)=\sum\limits_{\rho\in\Delta(\CL_e^\star(\G,F))/\G^F}(-1)^{|\rho|+1}\k^d(B_\rho,\lambda).
\end{equation}
Next, observe that since our $\ell$-block $B$ has positive defect it cannot contain a character of $\ell$-defect zero (see, for instance, \cite[Theorem 3.18]{Nav98}). Then, by applying Lemma \ref{lem:Defect zero characters and cuspidality}, we deduce that the $\ell$-block $B$ does not contain any $e$-cuspidal character, so that $\k_{\rm c}^d(B,\lambda)=0$. This shows that \eqref{eq:Replacement theorem for Dade 1} is equivalent to the equality stated in Conjecture \ref{conj:Dade reductive}, and therefore we conclude that (i) and (ii) are equivalent.
\end{proof}

As an immediate consequence of the above result we get Theorem \ref{thm:Main replacement theorem for Dade}.

\begin{proof}[Proof of Theorem \ref{thm:Main replacement theorem for Dade}]
We now assume that $\G$ is a simple algebraic group of simply connected type and that $\ell$ is an odd prime number good for $\G$ and not dividing the order of $\z(\G)^F$, with $\ell\neq 3$ if $(\G,F)$ has a rational component of type ${^3{\bf D}_4}$. Then surely $\ell\in\pi(\G,F)\cap \pi(\G^*,F^*)$. Furthermore, since $\G^*$ is of adjoint type, \cite[Proposition 2.4.4]{Dig-Mic20} implies that $\z(\G^*)=1$ and hence the hypothesis of Theorem \ref{thm:Replacement theorem for Dade} is satisfied. The result now follows from Theorem \ref{thm:Replacement theorem for Dade} by considering the trivial subgroup $Z=1$, with its trivial character $\lambda=1_Z$, and noticing that $\k^d(B_\sigma)=\k^d(B_\sigma,\lambda)$ for every $\sigma$ belonging to $\Delta(\mathcal{S}_\ell(\G^F))$ or to $\Delta(\CL_e(\G,F))$.
\end{proof}

We now come to the study of the Character Triple Conjecture. In the next proposition we show that, under suitable hypotheses, the pairs $(\sigma,\vartheta)$ of $\C^d(B)$, for an $\ell$-block $B$ of $\G^F$ and a non-negative integer $d$, can be replaced with analogous pairs in which we consider chains belonging to the simplicial complex of $e$-split Levi subgroups $\Delta(\CL_e(\G,F))$. For this purpose, let us define the set $\C^d_e(B)$ consisting of those pairs $(\rho,\chi)$ where $\rho$ belongs to $\Delta(\CL_e(\G,F))$ and $\chi$ is an irreducible character of the stabiliser $\G_\rho^F$ with $\ell$-defect $d$ and such that, if $\bl(\chi)$ denotes the $\ell$-block of $\G_\rho^F$ containing $\chi$, the induced $\ell$-block $\bl(\chi)^{\G^F}$ is defined and coincides with $B$. Observe that it is not clear, in general, whether block induction from $\G_\rho^F$ to $\G^F$ is defined, however in the results presented below this will always be the case.

\begin{prop}
\label{prop:Replacement theorem for CTC}
Let $(\G,F)$ be a finite reductive group and consider a prime $\ell\in\pi(\G,F)$ not dividing the order of $\z(\G)^F$. Fix an $\ell$-block $B$ of $\G^F$ and a non-negative integer $d$. Then there exist $\aut_\mathbb{F}(\G^F)_B$-equivariant bijections
\[\Psi_\pm:\C^d_{\gamma_{\ell,e}}(B)_\pm/\G^F\to \C^d_e(B)_\pm/\G^F\]
such that
\[\left(X_{\sigma,\vartheta},\G^F_\sigma,\vartheta\right)\iso{\G^F}\left(X_{\rho,\chi},\G^F_\rho,\chi\right)\]
for every $(\sigma,\vartheta)\in\C^d_{\gamma_{\ell,e}}(B)_\pm$ and $(\rho,\chi)\in\Psi_\pm(\overline{(\sigma,\vartheta)})$ and where $X:=\G^F\rtimes \aut_\mathbb{F}(\G^F)$.
\end{prop}

\begin{proof}
Since our assumption ensures that $\z(\G)^F_\ell$ is trivial, we can apply Lemma \ref{lem:e-split and e-closed subgroups} to obtain a bijection $\delta$ between the simplicial complexes $\Delta(\ab_\ell(\G^F)^{\gamma_{\ell,e}})$ and $\Delta(\CL_e(\G,F))$. Consider now $\sigma\in\Delta(\ab_\ell(\G^F)^{\gamma_{\ell,e}})$ and $\rho\in\Delta(\CL_e(\G,F))$ with $\delta(\sigma)=\rho$. By the properties of $\delta$ we deduce that these two chains have the same stabiliser $\G_\sigma^F=\G_\rho^F$ and the same length $|\sigma|=|\rho|$. In particular, this shows that for every irreducible character $\chi$ of $\G_\rho^F$ the $\ell$-block $\bl(\chi)$ induces a block of $\G^F$ according to \cite[Lemma 3.2]{Kno-Rob89}. Furthermore, it follows that $(\sigma,\vartheta)$ belongs to $\C^d_{\gamma_{\ell,e}}(B)_\pm$ if and only if $(\rho,\vartheta)$ belongs to $\C^d_e(B)_\pm$ while, by noticing that $X_\sigma=X_\rho$, we obtain
\[\left(X_{\sigma,\vartheta},\G^F_\sigma,\vartheta\right)\iso{\G^F}\left(X_{\rho,\vartheta},\G^F_\rho,\vartheta\right).\]
Then, using once again the equivariance properties of $\delta$, we can define the map $\Psi_\pm$ by sending the $\G^F$-orbit of the pair $(\sigma,\vartheta)\in\C^d_{\gamma_{\ell,e}}(B)_\pm$ to the $\G^F$-orbit of the pair $(\delta(\sigma),\vartheta)\in\C^d_e(B)_\pm$. As explained above, the bijections $\Psi_\pm$ satisfy the required conditions.
\end{proof}

By applying Theorem \ref{thm:Cancellation theorem for CTC} and Proposition \ref{prop:Replacement theorem for CTC} we can show that Conjecture \ref{conj:CTC reductive} implies the Character Triple Conjecture \ref{conj:CTC}. To do so, we use results on $e$-Harish-Chandra theory for $\ell$-singular series introduced in \cite{Ros-Generalized_HC_theory_for_Dade} that require the assumption \cite[Condition 3.3]{Ros-Generalized_HC_theory_for_Dade}. The latter condition provides a description of the irreducible constituents of Luszitg induction in terms of the order relation $\ll_e$ (see \cite[3.5.24]{Gec-Mal20}) and is related to a conjecture of Cabanes and Enguehard (see \cite[Notation 1.11]{Cab-Eng99} and \cite[Conjecture 3.2]{Ros-Generalized_HC_theory_for_Dade}) which is inspired by \cite[Theorem 3.11]{Bro-Mal-Mic93}. For the purpose of this paper, it is enough to know that \cite[Condition 3.3]{Ros-Generalized_HC_theory_for_Dade} holds for simply connected groups whenever $e=e_{\ell}(q)$ and $\ell\geq 5$ is good for $\G$ (see \cite[Corollary 3.7]{Ros-Generalized_HC_theory_for_Dade}). We refer the reader to \cite[Section 3]{Ros-Generalized_HC_theory_for_Dade} for further details.

\begin{theo}
\label{thm:Replacement theorem for CTC}
Let $(\G,F)$ be finite reductive group and consider a prime $\ell\in\pi(\G,F)\cap \pi(\G^*,F^*)$ not dividing the order of $\z(\G)^F$ nor that of $\z(\G^*)^{F^*}$. Assume further that $\ell\geq 5$ and that \cite[Condition 3.3]{Ros-Generalized_HC_theory_for_Dade} holds for $(\G,F)$. If Conjecture \ref{conj:CTC reductive} holds for an $\ell$-block $B$ of $\G^F$ with non-trivial defect and a non-negative integer $d$, then the Character Triple Conjecture \ref{conj:CTC} holds for $B$ and $d$ with respect to the inclusion $\G^F\unlhd \G^F\rtimes \aut_\mathbb{F}(\G^F)$.
\end{theo}

\begin{proof}
By assumption we know that Conjecture \ref{conj:CTC reductive} holds for $B$ and $d$, hence there exists a bijection $\Lambda$ between the $\G^F$-orbits in $\CL_e^d(B)_+$ and those in $\CL_e^d(B)_-$ that is $\aut_\mathbb{F}(\G^F)_B$-equivariant and induces $\G^F$-block isomorphisms of character triples. We want to use $\Lambda$ to construct a bijection $\Omega$ satisfying the conditions required by the Character Triple Conjecture \ref{conj:CTC}. To start, consider a pair $(\sigma,\vartheta)\in\C^d_e(B)_+$ and denote by $\L$ the smallest term of $\sigma$. Before proceeding further, we point out that the results obtained by assuming Hypothesis 4.1 of \cite{Ros-Generalized_HC_theory_for_Dade} are still valid under our assumption and by letting Lusztig induction (eventually) depend on the choice of a parabolic subgroup. We can therefore apply \cite[Lemma 5.5]{Ros-Generalized_HC_theory_for_Dade} to obtain an $e$-cuspidal pair $(\M,\mu)$ of $(\G,F)$ with $\M\leq \L$ and such that $\vartheta$ lies above some character of the $e$-Harish-Chandra series $\E(\L^F,(\M,\mu))$. Next, observe that $(\M,\mu)$ belongs to the set $\CP_e^\star(B)$. In fact, since $B$ has non-trivial defect, Lemma \ref{lem:Defect zero characters and cuspidality} ensures that $\CP_e(B)=\CP^\star_e(B)$ while, by applying \cite[Proposition 4.8 and Lemma 5.5]{Ros-Generalized_HC_theory_for_Dade}, we can use the condition $\bl(\vartheta)^{\G^F}=B$ to show that $(\M,\mu)$ belongs to $\CP_e(B)$. In particular, this shows that $(\sigma,\M,\ab(\mu),\vartheta)$ is an element of $\CL_e^d(B)_+$. Now, if $\Lambda$ maps the $\G^F$-orbit of $(\sigma,\M,\ab(\mu),\vartheta)\in\CL_e^d(B)_+$ to that of $(\sigma',\M',\ab(\mu'),\vartheta')\in\CL_e^d(B)_-$, arguing as above we conclude that $(\sigma',\vartheta')$ belongs to $\C^d_e(B)_-$. By mapping the $\G^F$-orbit of $(\sigma,\vartheta)$ to that of $(\sigma',\vartheta')$, we obtain an $\aut_\mathbb{F}(\G^F)_B$-equivariant bijection
\[\Phi:\C_e^d(B)_+/\G^F\to\C_e^d(B)_-/\G^F\]
that induces $\G^F$-block isomorphisms of character triples (this is because the $\G^F$-block isomorphisms induced by $\Lambda$ only depend on $(\sigma,\vartheta)$ and $(\sigma',\vartheta')$). Next, let $\Psi_\pm$ be the bijections introduced in Proposition \ref{prop:Replacement theorem for CTC} and define the map
\[\Theta:\C_{\gamma_{\ell,e}}^d(B)_+/\G^F\to \C_{\gamma_{\ell,e}}^d(B)_-/\G^F\]
given by $\Theta:=\Psi_-^{-1}\circ\Phi\circ\Psi_+$. It is immediate to show that $\Theta$ is $\aut_\mathbb{F}(\G^F)_B$-equivariant and, using the transitivity of $\iso{\G^F}$ (see \cite[Lemma 3.8 (a)]{Spa17}), that it induces $\G^F$-block isomorphisms of character triples. Finally, let $\Upsilon$ be the map given by Theorem \ref{thm:Cancellation theorem for CTC} and define
\[\Omega\left(\overline{(\sigma,\vartheta)}\right):=\begin{cases}
\Theta\left(\overline{(\sigma,\vartheta)}\right), &\text{if }(\sigma,\vartheta)\in\C_{\gamma_{\ell,e}}^d(B)_+
\\
\Upsilon\left(\overline{(\sigma,\vartheta)}\right), &\text{if }(\sigma,\vartheta)\in\C_{\gamma_{\ell,e}}^d(B)_+^{\rm c}.
\end{cases}\]
Since $\C^d(B)_\pm$ is partitioned into the $\aut_\mathbb{F}(\G^F)_B$-invariant subsets $\C_{\gamma_{\ell,e}}^d(B)_\pm$ and $\C_{\gamma_{\ell,e}}^d(B)_\pm^{\rm c}$, we deduce that $\Omega$ defines an $\aut_\mathbb{F}(\G^F)_B$-equivariant bijection between $\C^d(B)_+/\G^F$ and $\C^d(B)_-/\G^F$. Furthermore, $\Omega$ satisfies the condition on $\G^F$-block isomorphisms of character triples since this is also satisfied by $\Theta$ and $\Upsilon$.
\end{proof}

As a corollary of the above theorem we get Theorem \ref{thm:Main replacement theorem for inductive Dade's condition} that shows how to recover the inductive condition for Dade's Conjecture from Conjecture \ref{conj:CTC reductive}.

\begin{proof}[Proof of Theorem \ref{thm:Main replacement theorem for inductive Dade's condition}]
We now assume that $\G$ is a simple algebraic group of simply connected type such that $\G^F$ is the universal covering group of $\G^F/\z(\G)^F$. Suppose further that $\ell\geq 5$ is good for $\G$ and does not divide the order of $\z(\G)^F$. In this case, \cite[Condition 3.3]{Ros-Generalized_HC_theory_for_Dade} holds for $(\G,F)$ according to \cite[Corollary 3.7]{Ros-Generalized_HC_theory_for_Dade} (observe that the proof of this result can be carried out without assuming the Mackey formula simply by allowing the possibility that Lusztig induction might depend on the choice of a parabolic subgroup). Then, we can apply Theorem \ref{thm:Replacement theorem for CTC} to obtain a bijection $\Omega$ with the properties described in the Character Triple Conjecture \ref{conj:CTC} and with respect to the inclusion $\G^F\unlhd \G^F\rtimes \aut_\mathbb{F}(\G^F)=:X$. In particular, we have
\[\left(X_{\sigma,\vartheta},\G^F_\sigma,\vartheta\right)\iso{\G^F}\left(X_{\rho,\chi},\G_\rho^F,\chi\right)\]
for every $(\sigma,\vartheta)\in\C^d(B)_+$ whose $\G^F$-orbit corresponds to that of $(\rho,\chi)\in\C^d(B)_-$ under the map $\Omega$. Now, \cite[Lemma 3.4]{Spa17} implies that the centre of the restricted character $\vartheta_{\z(\G)^F}$ coincides with that of $\chi_{\z(\G)^F}$. If we denote this group by $Z$, then we deduce that its order is prime to $\ell$ and therefore \cite[Corollary 4.5]{Spa17} yields
\[\left(X_{\sigma,\vartheta}/Z,\G^F_\sigma/Z,\overline{\vartheta}\right)\iso{\G^F/Z}\left(X_{\rho,\chi}/Z,\G_\rho^F/Z,\overline{\chi}\right)\]
where we denote by $\overline{\vartheta}$ and $\overline{\chi}$ the characters corresponding to $\vartheta$ and $\chi$ via inflation. To conclude, observe that $\aut_\mathbb{F}(\G^F)$ induces all automorphisms of the finite group $\G^F$ (see \cite[1.15]{GLS}) and apply \cite[Lemma 7.6]{Ros-Generalized_HC_theory_for_Dade} to obtain the inductive condition for Dade's Conjecture for the $\ell$-block $B$ and the non-negative integer $d$.
\end{proof}

\bibliographystyle{alpha}
\bibliography{References}

\vspace{1cm}

Mathematisches Forschungsinstitut Oberwolfach
\\
Schwarzwaldstr. 9-11,
\\
77709 Oberwolfach-Walke,
\\
Germany

\textit{Email address:} \href{mailto:damiano.rossi00@gmail.com}{damiano.rossi.math@gmail.com}

\end{document}